\documentclass[10pt]{amsart}

\usepackage{amssymb,amsmath,amscd,amsthm,amsfonts,wasysym,mathrsfs,amsxtra,multirow}
\input xy
\xyoption{all}

\newcommand{\OO}{\mathscr{O}}

\newcommand{\II}{\mathscr{I}}

\newcommand{\MM}{\mathcal{M}}

\newcommand{\Cfield}{\mathbb{C}}

\newcommand{\Spec}{\textnormal{Spec }}

\newcommand{\image}{\textnormal{im}\,}

\newcommand{\HHom}{\mathscr{H}om} 
\newcommand{\Hom}{\textnormal{Hom}}
\newcommand{\dimension}{\textnormal{dim}\,}

\newcommand{\Pic}{\textnormal{Pic}}

\newcommand{\supp}{\textnormal{supp}}

\newcommand{\Ext}{\textnormal{Ext}}
\newcommand{\EExt}{\mathscr{E}xt}

\newcommand{\al}{\alpha}

\newcommand{\Coh}{\textnormal{Coh}}

\newcommand{\Ap}{\mathcal{A}^p}

\newcommand{\arinj}{\ar@{^{(}->}}
\newcommand{\arsurj}{\ar@{->>}}

\newtheorem{theorem}{Theorem}[section]

\newtheorem{lemma}[theorem]{Lemma}
\newtheorem{pro}[theorem]{Proposition}

\newtheorem{coro}[theorem]{Corollary}

\theoremstyle{definition}

\theoremstyle{remark}
\newtheorem{remark}[theorem]{Remark}

\numberwithin{equation}{section}

\begin{document}

\title{Stability and Fourier-Mukai transforms on elliptic fibrations}

\author[Jason Lo]{Jason Lo}
\address{Department of Mathematics, University of Missouri,
Columbia, MO 65211, USA} \email{jccl@alumni.stanford.edu}

\keywords{elliptic fibrations, stability, moduli, Fourier-Mukai transforms.}
\subjclass[2010]{Primary 14J60; Secondary: 14J27, 14J30}

\maketitle

\begin{abstract}
We systematically develop Bridgeland's \cite{FMTes} and Bridgeland-Maciocia's \cite{BMef} techniques for studying elliptic fibrations, and identify criteria that ensure 2-term complexes are mapped to torsion-free sheaves under a Fourier-Mukai transform.  As an application, we construct an open immersion from a moduli  of stable complexes to a moduli of Gieseker stable sheaves on elliptic threefolds.  As another application, we give various 1-1 correspondences between fiberwise semistable torsion-free sheaves and codimension-1 sheaves on Weierstrass surfaces.
\end{abstract}


\section{Introduction}

Fourier-Mukai transforms have been used extensively to understand stable sheaves and their moduli.  We mention only a few works below, and refer to \cite{FMNT} for a more comprehensive survey on this subject.

One important problem on Calabi-Yau threefolds is the construction of stable sheaves.  In \cite{FMW}, Friedman-Morgan-Witten developed a technique for constructing stable sheaves on an elliptic fibration $X$, using the notion of spectral covers.  In their method, there is a 1-1 correspondence, via a Fourier-Mukai transform, between the stable sheaves on $X$ and line bundles supported on lower-dimensional subvarieties (namely, the spectral covers) of the Fourier-Mukai partner $\hat{X}$.  This allows us to translate moduli problems for sheaves on $X$ to those on a lower-dimensional variety, for which we have more tools at our disposal.  This aspect of Fourier-Mukai transforms is especially relevant to the conjectural duality between F-theory and heterotic strings (see \cite{CDFMR, HRP}, for instance).

On a broader scale, Fourier-Mukai transforms can be used to describe various moduli problems on a variety $X$ in terms of moduli problems on its Fourier-Mukai partner $\hat{X}$.  For instance, Bruzzo-Maciocia \cite{BM} showed that if $X$ is a reflexive K3 surface, then Hilbert schemes of points on $X$ are isomorphic to moduli spaces of  stable locally free sheaves on $\hat{X}$, with the isomorphism given by a Fourier-Mukai transform.  And in \cite{FMTes}, Bridgeland showed that if $X$ is a relatively minimal elliptic surface, then Hilbert schemes of points on $\hat{X}$ are birationally equivalent to moduli of stable torsion-free sheaves on $X$.  If $X$ is an elliptic threefold, then Bridgeland-Maciocia \cite{BMef} showed that any connected component of a complete moduli of rank-one torsion-free sheaves is isomorphic to a component of the moduli of stable torsion-free sheaves on $\hat{X}$.

Since Bridgeland's work on stability conditions on triangulated categories \cite{StabTC,SCK3} appeared, there has been a lot of focus on stable objects in the bounded derived category of coherent sheaves $D(X)$ of a variety $X$ - which are chain complexes of coherent sheaves on $X$ - and their moduli spaces.  These moduli spaces and the associated counting invariants have rich connections with mirror symmetry.  And now, using Fourier-Mukai transforms, we can translate moduli problems for complexes on $X$ to moduli problems for sheaves on $\hat{X}$, the latter being better understood.  Recent works along this line include: Bernardara-Hein \cite{BH} and Hein-Ploog \cite{HP} for  elliptic K3 surfaces, Maciocia-Meachan \cite{MM} for rank-one Bridgeland stable complexes on Abelian surfaces,  Minamide-Yanagida-Yoshioka \cite{MYY,MYY2} for Bridgeland stable complexes on Abelian and K3 surfaces, and the author \cite{Lo4} for K3 surfaces.

By using Fourier-Mukai transforms to construct open immersions or isomorphisms from moduli of complexes to moduli of sheaves, we can use existing results on moduli of sheaves to better understand moduli of complexes, such as computing their counting invariants,  showing that they are fine moduli spaces, or showing they are birationally equivalent to other moduli spaces.

\subsection{Overview of results}

In this paper, we systematically develop the ideas originally found in Bridgeland's \cite{FMTes} and Bridgeland-Maciocia's \cite{BMef} papers on  elliptic surfaces and elliptic threefolds.  Given an elliptic surface or elliptic threefold $X$, the idea is to use three different torsion pairs $(\mathcal T_X, \mathcal F_X), (W_{0,X},W_{1,X})$ and $(\mathcal B_X, \mathcal B_X^\circ)$ (see Section \ref{section-notation} for their definitions) to break up the category $\Coh (X)$ into various subcategories, and understand how each category changes under the Fourier-Mukai transform from $X$; this is done in Section \ref{section-equiv}.  Our first key technical result is Theorem \ref{pro5}, which roughly says, that given a WIT$_1$ torsion-free sheaf $F$ on $X$ that restricts to a stable sheaf on the generic fibre, it is taken to a torsion-free sheaf if and only if it satisfies the vanishing condition
\[
  \Ext^1_{D(X)} (\mathcal B_X \cap W_{0,X},F)=0.
\]
Applying this criterion on elliptic threefolds, we construct an open immersion from a moduli stack of polynomial stable complexes $X$ to a moduli stack of stable sheaves on $\hat{X}$ in Theorem \ref{theorem-main2}.  Since the moduli of stable sheaves admits a tame moduli space in the sense of Alper \cite{Alper},  we obtain an example of a moduli of  complexes that  also admits a tame moduli space.

We also point out in Remark \ref{remark2} that, when $X$ is an elliptic surface, the birational equivalence from a moduli of sheaves on $X$ to $\text{Pic}^\circ (\hat{X}) \times \text{Hilb}^t(\hat{X})$ constructed in \cite[Theorem 1.1]{FMTes} restricts to an isomorphism precisely at the locus defined by the vanishing condition above.  Besides, all the sheaves parametrised by this locus are locally free.

By considering a category slightly larger than the image of the open immersion in Theorem \ref{theorem-main2}, we obtain an equivalence of categories on elliptic threefolds in Theorem \ref{theorem-main3}, between a category $\mathcal C_X$ of 2-term complexes on $X$ and a category of torsion-free sheaves $\mathcal C_{\hat{X}}$ on $\hat{X}$.  This equivalence not only extends the aforementioned open immersion, but also extends the isomorphism between a moduli of rank-one torsion-free sheaves and a moduli of stable torsion-free sheaves in \cite[Theorem 1.4]{BMef}.  

Finally, in Section \ref{section-app2}, we consider torsion-free sheaves on $X$ that are taken to sheaves supported in codimension-1 on $\hat{X}$.  On elliptic surfaces that are Weierstrass, we obtain an equivalence of categories between fiberwise locally free sheaves of degree 0 on $X$ and pure 1-dimensional sheaves flat over the base (Proposition \ref{pro4}).  As a special case, we have a 1-1 correspondence between line bundles of fibre degree 0 on a Weierstrass surface $X \to S$, and line bundles supported on sections of the dual fibration $\hat{X} \to S$ (Corollary \ref{coro2}).  These results resemble some of the results obtained using Friedman-Morgan-Witten's spectral construction, such as in \cite{HRP}, but do not make use of Fitting ideals.  It would be interesting to understand the precise connections between our results in Section \ref{section-app2} and those obtained using the spectral construction.

\subsection{Notation}
For any noetherian scheme $X$, we let $\Coh (X)$ denote the category of coherent sheaves on $X$, and $D(X)$ denote the bounded derived category of coherent sheaves on $X$.  For any $E \in D(X)$, we write $H^i(E)$ to denote the cohomology of $E$ at degree $i$.  If the dimension of $X$ is $n$ and $0\leq d \leq n$ is an integer, we write $\Coh_{\leq d}(X)$ to denote the subcategory of $\Coh (X)$ consisting of sheaves supported in dimension at most $d$, and write $\Coh_{\geq d}(X)$ to denote the subcategory of $\Coh (X)$ consisting of sheaves without subsheaves in $\Coh_{\leq d-1}$.

If $(\mathcal T, \mathcal F)$ is a torsion pair in $\Coh (X)$, we write $\langle \mathcal T, \mathcal F [1]\rangle$ to denote the extension-closed subcategory of $D(X)$ generated by $\mathcal T$ and $\mathcal F[1]$.  That is, elements $E$ in $\langle \mathcal T, \mathcal F [1]\rangle$ are exactly the complexes in $D(X)$ such that $H^0 (E) \in \mathcal T, H^{-1} (E) \in \mathcal F$ and $H^i (E)=0$ for any $i \neq -1, 0$.

Given varieties $X$ and $Y$, a functor $D(X) \to D(Y)$ of the form $$\Psi (-) := R{\pi_Y}_\ast (P \overset{L}{\otimes} \pi_X^\ast  (-) )$$ for some $P \in D(X \times Y)$ is called an integral functor.  Here, $\pi_X, \pi_Y$ denote the projections from $X \times Y$ onto $X,Y$, respectively.  We will use the term `Fourier-Mukai transform' only for integral functors that induce equivalences of categories.

\subsection{A review of polynomial stability conditions}

The reader may refer to Bayer's article \cite{BayerPBSC} for a complete explanation of polynomial stability conditions.  Here, we include only a brief summary.

Suppose $X$ is a smooth projective threefold.  A polynomial stability on $D(X)$ in the sense of Bayer is the data $\sigma = (\omega, \rho, p, U)$ where  $\omega$ is a fixed ample $\mathbb{R}$-divisor on $X$, whereas
\[
\rho = (\rho_0, \rho_1, \rho_2, \rho_3) \in (\Cfield^\ast)^{4}
 \]
 is a quadruple of nonzero complex numbers such that each $\rho_d/\rho_{d+1}$ lies in the upper half complex plane.  Also, $p$ is a perversity function associated to $\rho$, i.e.\ $p$ is a function $\{0,1,2,3\} \to \mathbb{Z}$ such that $(-1)^{p(d)}\rho_d$ lies in the upper half plane for each $d$.  The last part, $U$, of the data $\sigma$ is a unipotent operator (i.e.\ an element of $A^\ast (X)_\Cfield$ of the form $U=1+N$, where $N$ is concentrated in positive degrees).  The perversity function $p$ determines a t-structure on $D^b(X)$ with heart $\Ap$.  Once the data $\sigma$ is given, the group homomorphism (also called the `central charge')
 \begin{align*}
   Z_\sigma :  K(D^b(X)) &\to \Cfield [m] \\
     E &\mapsto Z_\sigma(E)(m) := \int_X \sum_{d=0}^3 \rho_d \omega^d m^d ch(E) \cdot U
 \end{align*}
 has the property that $Z_\sigma(E)(m)$ lies in the upper half plane for any $0 \neq E \in \Ap$ and real number $m \gg 0$.

 For $0\neq E \in \Ap$, if we write $Z_\sigma(E)(m) \in \mathbb{R}_{>0} \cdot e^{i \pi \phi (E) (m)}$ for some real number $\phi (E)(m)$ for $m \gg 0$, then we have $\phi (E)(m) \in (0, 1]$ for $m \gg 0$.  We say that $E$ is $\sigma$-semistable if, for all subobjects $0 \neq F \subsetneq E$ in $\Ap$, we have $\phi (F)(m) \leq \phi (E) (m)$ for all $m \gg 0$ (which we write $\phi (F) \preceq \phi (E)$ to denote);  and we say $E$ is $\sigma$-stable if $\phi (F) (m) < \phi (E)(m)$ for all $m \gg 0$ (which we write $\phi (F) \prec \phi (E)$ to denote).

{\bf Acknowledgments}:
The author would like to thank Zhenbo Qin for many enlightening discussions, and assistance with a key step in this project.  He would also like to thank Arend Bayer, Emanuel Diaconescu, Jun Li and Ziyu Zhang for helpful discussions, and Andrei ~C\u{a}ld\u{a}raru for answering his questions on elliptic threefolds.  Finally, he would like to thank the referees for their many careful comments, which have helped  improve the writing of the paper to a great extent.

\section{Complexes and Fourier-Mukai transforms}

The goal of this section is to find sufficient conditions for stable 2-term complexes to be mapped to  stable torsion-free sheaves by the Fourier-Mukai transform constructed in \cite{BMef}.

Given a Fourier-Mukai transform $\Psi : D(X) \to D(Y)$ between two derived categories, recall that a complex $E \in D(X)$ is called $\Psi$-WIT$_i$ if $\Psi (E)$ is isomorphic to an object in $\Coh (Y) [-i]$, i.e.\ the cohomology of $\Psi (E)$ vanishes at every degree $j$ that is not equal to $i$.  If $E$ is a sheaf on $X$, we consider $E$ as a complex concentrated in degree 0, and speak of $E$ being $\Psi$-WIT$_i$ in the above sense.  For a $\Psi$-WIT$_i$ complex $E \in D(X)$, we write $\hat{E}$ to denote the sheaf on $Y$, unique up to isomorphism, such that $\Psi (E)\cong \hat{E} [-i]$, and call $\hat{E}$ the transform of $E$.

\subsection{Outline of strategy}\label{section-BMefsum}

Suppose $\pi : X \to S$ and $\hat{\pi} : Y \to S$ are two elliptic threefolds over a surface $S$, and $\Psi : D(X) \to D(Y)$ is a relative Fourier-Mukai transform over $S$ in \cite{BMef}.  Before we explain our strategy for mapping stable 2-term complexes on $X$ to stable sheaves on $Y$, we review Bridgeland-Maciocia's approach in \cite{BMef} for mapping rank-one torsion-free sheaves on $X$ to stable sheaves on $Y$:
\begin{itemize}
\item[Step 1.] Given any rank-one torsion-free sheaf $F$ on $X$,  first twist $F$ by a high enough power of an ample line bundle $L$.  If $n \gg 0$, then $F \otimes L^{\otimes n}$ is $\Psi$-WIT$_0$ \cite[Corollary 8.5]{BMef}.  Note that the operation $F \mapsto F \otimes L^{\otimes n}$ does not alter the moduli space of rank-one torsion-free sheaves on $X$.
\item[Step 2.] Show that, for any $\Psi$-WIT$_0$ torsion-free sheaf $F$ on $X$, the transform $\Psi (F)$ is also a torsion-free sheaf \cite[Lemma 9.4]{BMef}.
\item[Step 3.] Show that, if a $\Psi$-WIT$_0$ torsion-free sheaf $F$ on $X$ restricts to a stable sheaf on the generic fibre of $\pi$, then the transform $\hat{F}$ restricts to a stable sheaf on the generic fibre of $\hat{\pi}$ \cite[Lemma 9.5]{BMef}.
\item[Step 4.] Show that, if a torsion-free sheaf $G$ on $Y$ restricts to a stable sheaf on the generic fibre of $\hat{\pi}$, then $G$ is stable with respect to a suitable polarisation on $Y$ \cite[Lemma 2.1]{BMef}.
\end{itemize}
Since any rank-one torsion-free sheaf on $X$ restricts to a stable sheaf on the generic fibre of $\pi$, and $\Psi$ preserves families of $\Psi$-WIT$_i$ sheaves, Steps 1 through 4 above imply that the any connected component $\mathcal N$ of the moduli of rank-one torsion-free sheaves on $X$ can be embedded into a connected component $\mathcal M$ of the moduli of stable sheaves on $Y$; if $\mathcal N$ is complete, then it is mapped isomorphically onto $\mathcal M$ \cite[Theorem 1.4]{BMef}.

Let $X, Y$ and $\Psi$ be as above.  Consider 2-term complexes $E \in D(X)$ concentrated in degrees 0 and $-1$ of the following form:
\begin{itemize}
\item $H^{-1}(E)$ is a torsion-free sheaf;
\item $H^0(E)$ is a sheaf supported in dimension 0.
\end{itemize}
We consider these complexes partly because moduli spaces of complexes of this form on threefolds have been constructed in \cite{Lo2, Lo3}.  For the Fourier-Mukai transform $\Psi$, 0-dimensional sheaves on $X$ are always $\Psi$-WIT$_0$.  Since the image of any coherent sheaf on $X$ under $\Psi$ is a complex with nonzero cohomology only perhaps at degrees 0 and 1, from the canonical exact triangle
\begin{equation*}
  H^{-1}(E)[1] \to E \to H^0(E) \to H^{-1}(E)[2],
\end{equation*}
we see that the transform $\Psi (E)$ of $E$ is isomorphic to a sheaf if and only if $H^{-1}(E)$ is $\Psi$-WIT$_1$.  When this is the case, $E$ is $\Psi$-WIT$_0$ and the exact triangle above is taken to the short exact sequence of coherent sheaves
\begin{equation*}
  0 \to \widehat{H^{-1}(E)} \to \hat{E} \to \widehat{H^0(E)} \to 0.
\end{equation*}
Once we know that $E$ is $\Psi$-WIT$_0$, if we want the transform $\hat{E}$ to be a stable sheaf, we must ensure that it is torsion-free to begin with.  In particular, this requires the transform $\widehat{H^{-1}(E)}$ to be a  torsion-free sheaf.  We therefore need to find a criterion that ensures the transform of a WIT$_1$ torsion-free sheaf is again torsion-free.  This is where  the case of complexes  has to depart from the case of sheaves (Step 2 above).  Finding such a criterion  will be the main goal of Section \ref{section-equiv}.

\subsection{Elliptic curves}\label{section-ellipcurv}

Many properties of Fourier-Mukai transforms on elliptic fibrations are similar to properties of Fourier-Mukai transforms on elliptic curves, which are well-understood - see \cite[Section 3.5.1]{FMNT}, for instance.

Suppose $X$ is an elliptic curve, $\hat{X}=\text{Pic}^0(X)$ is the dual variety, and $\mathcal P$ is the Poincar\'{e} line bundle on $X \times \hat{X}$.  Let $\Psi$ denote the Fourier-Mukai transform $D^b(X) \to D^b(\hat{X})$ with kernel $\mathcal P$.  For any $\beta \in \mathbb R$, we can define the following full subcategories of $\Coh (X)$:
\begin{itemize}
\item $\mathcal T_X (\beta)$ is the category of coherent sheaves $E$ where all the Harder-Narasimhan (HN) factors are either torsion, or have slopes $\mu > \beta$.
\item $\mathcal F_X (\beta)$ is the category of coherent sheaves $E$ where all the HN factors are torsion-free, and have slopes $\mu \leq \beta$.
\end{itemize}
Then $(\mathcal T_X (\beta), \mathcal F_X (\beta))$ is a torsion pair in $\Coh (X)$.  Let $\mathcal A_X (\beta)$ denote the heart obtained by tilting $\Coh (X)$ with respect to this torsion pair.  That is, $\mathcal A_X (\beta) = \langle \mathcal F_X(\beta ) [1], \mathcal T_X (\beta)\rangle$.  We have the following lemma, which seems well-known (see \cite{BayerTour}, for instance), but whose proof we include here for ease of reference:

\begin{lemma}\label{lemma4}
The Fourier-Mukai transform $\Psi : D^b(X) \to D^b(\hat{X})$ induces an equivalence of Abelian categories $\mathcal A_X (0) \to \Coh (\hat{X})$.
\end{lemma}

\begin{proof}
Recall the well-known result, that if $\mathcal H_1, \mathcal H_2$ are two hearts of bounded t-structures of the derived category of an Abelian category with $\mathcal H_1 \subseteq \mathcal H_2$, then necessarily $\mathcal H_1 = \mathcal H_2$.  Therefore, it suffices for us to show that $\Psi$ takes $\mathcal A_X(0)$ into $\Coh (\hat{X})$.  Moreover, since the category $\Coh (\hat{X})$ is extension-closed, it suffices to show that, for any slope stable coherent sheaf $F$ on $X$, we have either $\Psi (F) \in \Coh (\hat{X})$ (if $F \in \mathcal A_X(0)$) or $\Psi (F[1]) \in \Coh (\hat{X})$ (if $F[1] \in \mathcal A_X (0)$).

Suppose $F$ is a slope stable coherent sheaf $F$ on $X$, and that $F \in \mathcal A_X (0)$.  Then $F$ has strictly positive slope, and is $\Psi$-WIT$_0$ by \cite[Corollary 3.29]{FMNT}; hence $\Psi (F) \in \Coh (\hat{X})$.  Now, suppose $F [1] \in \mathcal A_X (0)$ instead.  Then $F$ has non-positive slope, and is $\Psi$-WIT$_1$ by \cite[Corollary 3.29]{FMNT} again; therefore, $\Psi (F[1]) \in \Coh (\hat{X})$, and we are done.
\end{proof}

\subsection{More notation and preliminaries}\label{section-notation}

By an elliptic fibration, we  mean a flat morphism of smooth projective varieties $\pi : X \to S$ where the generic fibre is a smooth genus one curve, such that $K_X \cdot C =0$ for any curve $C$ contained in a fibre of $\pi$.  (In \cite[Definition 1.1]{BMef}, $\pi$ is not necessarily assumed to be flat or projective.)  When $X$ is a threefold and $S$ is a surface, we refer to $X$ or $\pi$ as an elliptic threefold; when $X$ is a surface and $S$ a curve, we refer to $X$ or $\pi$ as an elliptic surface.

In the rest of Section \ref{section-notation}, we give a brief summary of the results on elliptic fibrations due to Bridgeland \cite{FMTes} and Bridgeland-Maciocia \cite{BMef}.  To be consistent with the notation in \cite{BMef}, given an elliptic fibration $\pi : X \to S$, we write $f$ to denote the class of any fibre of $\pi$ in the Chow ring of $X$, i.e.\ the `fibre class' of $\pi$.  Then for any object $E \in D(X)$, we define the fibre degree of $E$ to be
\[
d(E) = c_1(E)\cdot f,
 \]
 which is the degree of the restriction of $E$ to the generic fibre of $\pi$.  For the rest of this article, for any coherent sheaf $E$, we write $r(E)$ to denote its rank, and when $r(E)>0$, we  define
 \[
 \mu (E) = d(E)/r(E),
 \]
 which is the  slope of the restriction of  $E$ to the generic fibre.   Let $\lambda_{X/S}$ denote the greatest common divisor of the fibre degrees of all objects in $D(X)$.

From \cite[Theorem 5.3]{FMTes} and \cite[Theorem 9.1]{BMef}, we know that given an elliptic threefold (resp.\ elliptic surface) $\pi : X \to S$ and  any element
 \[
  \begin{pmatrix} c & a \\ d & b \end{pmatrix} \in \text{SL}_2(\mathbb{Z})
 \]
 where $a>0$ and $\lambda_{X/S} |d$,  there is another elliptic threefold (resp.\ elliptic surface) $\hat{\pi} : Y \to S$ that is a relative moduli of sheaves on $X$, where  given any point $s\in S$, the fibre $Y_s$ is the moduli of stable sheaves of rank $a$ and degree $b$ on $X_s$ (e.g.\ see \cite[Section 4]{FMTes}, \cite[Section 2.1]{BMef} or \cite[Section 6.3]{FMNT}).  In fact, the situation is symmetric in $X$ and $Y$, in the sense that $X$ is also a relative moduli of sheaves on $Y$, where given any point $s \in S$, the fibre $X_s$ is the moduli of stable sheaves of rank $a$ and degree $c$ on $Y_s$ (see \cite[Lemma 5.2]{FMTes} or \cite[Proposition 8.7]{BMef}).

If $\pi : X \to S$ is an elliptic threefold or surface, and $Y$ is as above, then the pushforward of the universal sheaf $\mathcal P$ on $X \times_S Y$ to $X \times Y$ acts as  the kernel of  a Fourier-Mukai transform $\Phi : D(Y) \to D(X)$.  If we let $\mathcal Q = \mathcal P^\vee \otimes \pi_X^\ast \omega_X [n-1]$, where $n$ is the dimension of $X$ and $Y$, then $\mathcal Q$ is the kernel of another Fourier-Mukai transform $\Psi : D(X) \to D(Y)$, and $\Phi$ is the inverse to the functor $\Psi [1]$.  That is, we have isomorphisms of functors
 \begin{gather*}
   \Psi \circ \Phi  \cong \text{id}_Y [-1] \text{, \qquad} \Phi \circ \Psi  \cong \text{id}_X [-1].
 \end{gather*}
 For any complex $E \in D(X)$, we write $\Psi^i(E)$ to denote the cohomology $H^i(\Psi (E))$; if $E$ is a sheaf sitting at degree 0, we have that $\Psi^i (E)=0$ unless $0\leq i \leq 1$, i.e.\ $\Psi (E) \in D^{[0,1]}_{\Coh (X)}(X)$.  The same statements hold for $\Phi$ and $Y$.

We also have the following formulas for how rank and fibre degree change under the Fourier-Mukai transforms $\Psi$ and $\Phi$:
\begin{align}
  \begin{pmatrix} r(\Phi E) \\ d(\Phi E) \end{pmatrix} &=
  \begin{pmatrix} c & a \\ d & b \end{pmatrix} \begin{pmatrix} r(E) \\ d(E) \end{pmatrix}  \text{ for all $E \in D(Y)$,}\label{eqn-rdunderPhi}  \\
    \begin{pmatrix} r(\Psi E) \\ d(\Psi E) \end{pmatrix} &=
  \begin{pmatrix} -b & a \\ d & -c \end{pmatrix} \begin{pmatrix} r(E) \\ d(E) \end{pmatrix} \text{ for all $E \in D(X)$}. \label{eqn-rdunderPsi}
\end{align}

We define the following full subcategories of $\Coh (X)$, all of which are extension-closed:
\begin{align*}
 \mathcal T_X &= \{ \text{torsion sheaves on } X \} \\
  \mathcal F_X &= \{ \text{torsion-free sheaves on }X \} \\
  W_{0,X} &= \{ \Psi\text{-WIT}_0 \text{ sheaves on }X \}  \\
  W_{1,X} &= \{ \Psi\text{-WIT}_1 \text{ sheaves on }X \} \\
  \mathcal B_X &= \{ E \in \Coh (X): r(E)=d(E)=0 \} \\
    \Coh (X)_{r>0} &= \{ E \in \Coh (X) : r(E)>0\}.
\end{align*}
And for any $s \in \mathbb{R}$, we define
\begin{align*}
  \Coh (X)_{\mu > s} &= \{ E \in \Coh (X)_{r>0}: \mu (E) >s \} \\
  \Coh (X)_{\mu=s} &= \{ E \in \Coh (X)_{r>0}: \mu (E) =s \} \\
  \Coh (X)_{\mu<s} &= \{ E \in \Coh (X)_{r>0}: \mu (E) <s \}.
\end{align*}
We  define the corresponding full subcategories of $\Coh (Y)$ similarly.  Some relations between these categories are immediate from their definitions.  For instance, we have the torsion pairs $(\mathcal T_X, \mathcal F_X), (W_{0,X}, W_{1,X})$ in $\Coh (X)$.  That $(W_{0,X},W_{1,X})$ is a torsion pair in $\Coh (X)$ follows from \cite[Lemma 6.1]{FMTes} when $\pi$ is an elliptic surface, and the same proof applies when $\pi$ is a fibration of higher dimensions.  Similarly, we have the torsion pairs $(\mathcal T_Y, \mathcal F_Y), (W_{0,Y}, W_{1,Y})$ in $\Coh (Y)$.  Also, for $i=0,1$, we have $\Psi (W_{i,X})=W_{1-i,Y}[-i]$ while $\Phi (W_{i,Y}) = W_{1-i,X}[-i]$.

\subsection{Torsion pairs and equivalences}\label{section-equiv}

 The main goal of this section is to identify a criterion under which a torsion-free $\Psi$-WIT$_1$ sheaf has torsion-free transform, which is Theorem \ref{pro5}.  In Theorem \ref{theorem-main0}, we give a class of sheaves that satisfies this criterion.

 Unless otherwise stated, every result in this section holds regardless of whether $\pi : X \to S$ is an elliptic surface or an elliptic threefold.

Note that, for any torsion sheaf $T$ on $X$, $c_1(T) f \geq 0$, i.e.\ $d(T) \geq 0$.  It follows that $\mathcal B_X$ is closed under subobjects, quotients and extensions in the abelian category $\Coh (X)$, and so is a Serre subcategory of $\Coh (X)$.  Thus we have:

\begin{lemma}\label{lemma17}
If we define
\[
\mathcal B^\circ_X := \{ E \in\Coh (X): \Hom (\mathcal B_X, E)=0\},
 \]
 then $(\mathcal B_X, \mathcal B_X^\circ)$ is a torsion pair in $\Coh (X)$.
\end{lemma}

\begin{proof}
Since $\Coh (X)$ is a Noetherian abelian category, this follows directly from \cite[Lemma 1.1.3]{Pol}.
\end{proof}

Note that  $\mathcal B_X \subset \mathcal T_X$ and $\mathcal F_X \subset \mathcal B_X^\circ$.  Also, if $E$ is a WIT sheaf on $X$, then $E \in \mathcal B_X$ iff $\hat{E} \in \mathcal B_Y$ - see \eqref{eqn-rdunderPhi} and \eqref{eqn-rdunderPsi}.

The following lemma gives us another way to think about objects in the category $\mathcal B_X$:

\begin{lemma}\label{lemma-BXdesc}
A sheaf $E$ on $X$ is in $\mathcal B_X$ iff the restriction $E|_{\pi^{-1}(s)}=0$ for a generic fibre $\pi^{-1}(s)$ of $\pi$.
\end{lemma}
\begin{proof}
The `if' direction is clear.  For the `only if' direction, suppose $E \in \mathcal B_X$.  Hence $r(E)=0$.  If $\dimension \text{supp} (\pi_\ast E) = \dimension S$, then $d(E)$ would be positive, and so we must have $\dimension \text{supp}(\pi_\ast E) < \dimension S$.  That is, $E|_{\pi^{-1}(s)}=0$ for a generic fibre $\pi^{-1}(s)$ of $\pi$.
\end{proof}
 Given a sheaf $E$ on $X$, we will say $E$ is a fibre sheaf if it is supported on a finite number of fibres of $\pi$.

\begin{remark}
When $\pi$ is an elliptic surface, a sheaf $E$ is in $\mathcal B_X$ iff it is supported on a finite union of fibres of $\pi$ \cite[Section 4.1]{FMTes}.  When $\pi$ is an elliptic threefold, however, the same statement does not hold, because a sheaf $E$ in $\mathcal B_X$ could be supported in dimension 2, but with $\pi(\text{supp}(E))$ being 1-dimensional.
\end{remark}

Motivated by the proof of Lemma \ref{lemma4}, we try to understand the image of the category $\Coh (X)$ under $\Psi$ by considering the  intersections of the various torsion classes and torsion-free classes above, and understanding their images under $\Psi$.  By symmetry, all the results stated for $X$ in this section have their counterparts for $Y$, if we interchange the roles of $X$ and $Y$ (and $\Psi, \Phi$).

\begin{lemma}\cite[Lemma 6.2]{FMTes}\label{lemma-E3-muWIT}
Let $E$ be a sheaf of positive rank on $X$.  If $E$ is $\Psi$-WIT$_0$, then $\mu (E) \geq b/a$.  If $E$ is $\Psi$-WIT$_1$, then $\mu (E) \leq b/a$.
\end{lemma}
\begin{proof}
The surface case is already stated in \cite[Lemma 6.2]{FMTes}.  For elliptic threefolds, the proof goes through without change.
\end{proof}

\begin{lemma}\label{lemma-E3-WIT1torsion}
If $T$ is a $\Psi$-WIT$_1$ torsion sheaf on $X$, then $T \in \mathcal B_X$.
\end{lemma}
\begin{proof}
The surface case is \cite[Lemma 6.3]{FMTes}.  The following argument works for both surfaces and threefolds: we have $0 \geq r(\Psi T)=-b\cdot r(T)+a\cdot d(T)$.  Since $r(T)=0$ and $d(T) \geq 0$, we have $d(T)=0$.  Hence $T \in \mathcal B_X$ by definition.
\end{proof}

\begin{remark}  Given any $E \in D^b(X)$, we have $r(\Psi E) = -b\cdot r(E)+a\cdot d(E)$.  So when $E$ has positive rank, $\mu (E)=b/a$ is equivalent to $r(\Psi E)=0$.  In other words, if $E$ is a $\Psi$-WIT sheaf on $X$ of positive rank with $\mu (E)=b/a$, then $\hat{E}$ is a torsion sheaf on $Y$.
\end{remark}

The following lemma is slightly more specific than Lemma \ref{lemma-E3-muWIT}:

\begin{lemma}\label{lemma-PsiWIT0posrk}
Suppose $E$ is a $\Psi$-WIT$_0$ sheaf on $X$ and $r(E)>0$.  Then $\mu (E)>b/a$.
\end{lemma}

\begin{proof}
Suppose $E$ satisfies the assumptions, and $\mu (E)=b/a$. Then by the remark above, $\hat{E}$ is a $\Phi$-WIT$_1$ torsion sheaf, and hence lies in $\mathcal B_Y$ by Lemma \ref{lemma-E3-WIT1torsion}.   This implies $E$ itself is in $\mathcal B_X$,  contradicting $r(E)>0$.  Then, by Lemma \ref{lemma-E3-muWIT}, we must have $\mu (E) > b/a$.
\end{proof}

\begin{lemma}\label{lemma-cohomofFMTfibresheaf}
Suppose $T \in \mathcal B_X$.   Then $\Psi^0(T), \Psi^1(T)$ are both torsion sheaves.
\end{lemma}
\begin{proof}
Suppose $T \in \mathcal B_X$.  Then $r(\Psi T)=0=d(\Psi T)$.  Hence $\Psi^0(T)$ and $\Psi^1(T)$ have the same rank and fibre degree.  Suppose $\Psi^0(T)$ has positive rank.  Then $\Psi^1(T)$ also has positive rank, and $\mu (\Psi^0(T))=\mu (\Psi^1(T))$.  Since $\Psi^0(T)$ is $\Phi$-WIT$_1$ and $\Psi^1(T)$ is $\Phi$-WIT$_0$, by Lemma \ref{lemma-E3-muWIT}, we have $\mu (\Psi^0(T))=\mu (\Psi^1(T))=b/a$.  However, this means $\Psi^1(T)$ is $\Phi$-WIT$_0$, with positive rank and  $\mu (\Psi^1(T)) =b/a$, contradicting Lemma \ref{lemma-PsiWIT0posrk}.  Hence $\Psi^0(T), \Psi^1(T)$ must both be torsion sheaves.
\end{proof}

%


\begin{lemma}\label{lemma-equiv3}
We have an equivalence of categories
\begin{gather}\label{eq16}
\mathcal F_X \cap \{ E \in \Coh (X) : \Ext^1 (\mathcal B_X \cap W_{0,X}, E)=0\} \cap W_{1,X} \overset{\Psi [1]}{\to} \mathcal B_Y^\circ \cap W_{0,Y}.
\end{gather}
\end{lemma}

We single out two key steps in the proof of Lemma \ref{lemma-equiv3} as Lemmas \ref{lemma15} and  \ref{lemma16} below.  Lemma \ref{lemma15} is a  generalisation of  \cite[Lemma 9.4]{BMef}, which says that torsion-free WIT$_0$ sheaves have torsion-free transforms:

\begin{lemma}\label{lemma15}
Let $F$ be a $\Phi$-WIT$_0$ sheaf on $Y$.  Then $\hat{F}$ is a torsion-free sheaf on $X$ if and only if $\Hom (\mathcal B_Y \cap W_{0,Y},F)=0$.
\end{lemma}
\begin{proof}
Consider the short exact sequence $0 \to A \to \hat{F} \to B \to 0$ in $\Coh (X)$, where $A$ is the maximal torsion subsheaf of $\hat{F}$.  Since $\hat{F}$ is $\Psi$-WIT$_1$, so is $A$, and so $A \in \mathcal B_X \cap W_{1,X}$ by Lemma \ref{lemma-E3-WIT1torsion}.  On the other hand, we have $\Hom (\mathcal B_X \cap W_{1,X},\hat{F}) \cong \Hom (\mathcal B_Y \cap W_{0,Y},F)$.

Therefore, if $\Hom (\mathcal B_Y \cap W_{0,Y},F)=0$, then $A$ must be zero, i.e.\ $\hat{F}$ is torsion-free.  Conversely, if $\hat{F}$ is torsion-free, then because every sheaf in $\mathcal B_X \cap W_{1,X}$ is torsion, we have $\Hom (\mathcal B_X \cap W_{1,X},\hat{F}) =0$, and so $\Hom (\mathcal B_Y \cap W_{0,Y},F)=0$.  Thus the lemma holds.
\end{proof}

\begin{lemma}\label{lemma16}
Let $F$ be a $\Phi$-WIT$_0$ sheaf on $Y$.  Then
\[
  \Hom (\mathcal B_Y \cap W_{1,Y}, F) \cong \Ext^1 (\mathcal B_X \cap W_{0,X},\hat{F}).
\]
\end{lemma}
\begin{proof}
This follows from $\Phi$ being an equivalence.
\end{proof}

\begin{proof}[Proof of Lemma \ref{lemma-equiv3}]
Take any nonzero $F \in \mathcal B_Y^\circ \cap W_{0,Y}$.  Then $\Hom (\mathcal B_Y, F)=0$, and so by Lemma \ref{lemma15}, we know $\hat{F}$ is torsion-free, i.e.\ $\hat{F} \in \mathcal F_X$.  On the other hand, Lemma \ref{lemma16} implies that $\Ext^1 (\mathcal B_X \cap W_{0,X},\hat{F})=0$.  Hence $\hat{F}$ lies in the left-hand side of \eqref{eq16}.

For the other direction, take any nonzero sheaf $E$ belonging to the left-hand side of \eqref{eq16}. By Lemmas \ref{lemma15} and \ref{lemma16}, we get that
\[
\Hom (\mathcal B_Y\cap W_{0,Y},\hat{E})=0=\Hom (\mathcal B_Y \cap W_{1,Y},\hat{E}).
\]
Since $(\mathcal B_Y \cap W_{0,Y}, \mathcal B_Y \cap W_{1,Y})$ is a torsion pair in $\mathcal B_Y$ (see Lemma \ref{lemma17}), we get $\Hom (\mathcal B_Y,\hat{E})=0$, i.e.\ $\hat{E} \in \mathcal B_Y^\circ$. This completes the proof of the lemma.
\end{proof}

The following lemma for elliptic surfaces generalises to elliptic threefolds with the same proof:
\begin{lemma}\cite[Lemma 6.4]{FMTes}\label{lemma-FMTesLemma6-4}
Let $E$ be a torsion-free sheaf on $X$ such that the restriction of $E$ to a general fibre of $\pi$ is stable.  Suppose $\mu (E) < b/a$.  Then $E$ is $\Psi$-WIT$_1$.
\end{lemma}

Lemma \ref{lemma-FMTesLemma6-4} implies that, if $E$ is a torsion-free sheaf on $X$ such that its restriction to the generic fibre of $\pi$ is stable, then $E(m)$ is $\Psi$-WIT$_1$ for $m \ll 0$.  Lemma \ref{lemma23} shows, however, that  the $\Psi$-WIT$_1$ torsion-free sheaves we obtain this way do not always have torsion-free transforms.  This is in contrast with the case of $\Psi$-WIT$_0$ torsion-free sheaves, which always have torsion-free transforms whether $\pi$ is an elliptic surface or an elliptic threefold  (see \cite[Lemma 7.2]{FMTes} and \cite[Lemma 9.4]{BMef}).

\begin{lemma}\label{lemma-equiv2}
The functor $\Psi [1]$  restricts to an equivalence of categories
\begin{gather}\label{eq17}
W_{1,X} \cap \Coh (X)_{r >0} \cap \Coh (X)_{\mu < b/a} \overset{\Psi [1]}{\to} W_{0,Y} \cap \Coh (Y)_{r>0} \cap \Coh (Y)_{\mu > -c/a}.
\end{gather}
\end{lemma}

\begin{proof}
Take any nonzero $E$ in the left-hand side of \eqref{eq17}.    By \eqref{eqn-rdunderPsi},
\begin{equation}\label{eq18}
 r(\hat{E}) = -r(\Psi E) = b\cdot r(E) - a\cdot d(E),
\end{equation}
which is positive since $\mu (E) < b/a$.  Since $r(E)>0$, from \eqref{eqn-rdunderPhi} we have
\begin{equation*}
r(E) = r(\Phi \hat{E}) = c\cdot r(\hat{E}) + a\cdot d(\hat{E});
\end{equation*}
since $r(E)$ is positive, we obtain  $\mu (\hat{E})>-c/a$.  This shows that $\hat{E}$ lies in the category on the right-hand side of \eqref{eq17}.    The proof of the other direction is similar.
\end{proof}

\begin{lemma}\label{lemma-equiv1}
The functor $\Psi [1]$ restricts to an equivalence of categories
\begin{gather}\label{eq19}
  W_{1,X} \cap \Coh (X)_{r>0} \cap \Coh (X)_{\mu = b/a} \overset{\Psi [1]}{\to} W_{0,Y} \cap (\mathcal T_Y \setminus \mathcal B_Y).
\end{gather}
\end{lemma}

\begin{proof}
Take any nonzero $E$ belonging to the left-hand side of \eqref{eq19}.  From \eqref{eq18} and $\mu (E)=b/a$, we get $r(\hat{E})=0$, i.e.\ $\hat{E}$ is a torsion sheaf.  That $E$ is not in $\mathcal B_X$ implies $\hat{E}$ is not in $\mathcal B_Y$.  Hence $\hat{E}$ lies in the category on the right-hand side of \eqref{eq19}.

For the other direction, take any nonzero $E$ from the right-hand side of \eqref{eq19}.  From \eqref{eqn-rdunderPsi}, we have
\begin{equation}\label{eqn-lemma-equiv1}
0=r(E)=-r(\Psi \hat{E})=b\cdot r(\hat{E})-a\cdot d(\hat{E}).
  \end{equation}
  Note that $\hat{E}$ cannot be a torsion sheaf, for if it were,  it would be a $\Psi$-WIT$_1$ torsion sheaf, and hence lies in $\mathcal B_X$  by Lemma \ref{lemma-E3-WIT1torsion}.  Then $E$ itself would be in $\mathcal B_Y$, a contradiction.  Hence $r(\hat{E})>0$, and \eqref{eqn-lemma-equiv1} gives $\mu (\hat{E}) = b/a$.
\end{proof}

We have observed that  a $\Psi$-WIT$_1$ torsion sheaf on $X$ lies in $\mathcal B_X$ (Lemma \ref{lemma-E3-WIT1torsion}), and so its transform is necessarily in $\mathcal B_Y$.  This, together with Lemmas \ref{lemma-equiv2} and \ref{lemma-equiv1}, gives a complete description of the transforms  of all coherent sheaves in $W_{1,X}$ under $\Psi [1]$.

\begin{remark}\label{remark5}
Suppose $F$ is a coherent sheaf on $X$ satisfying:
 \begin{itemize}
 \item[(G)] $F$ is torsion-free, $\Psi$-WIT$_1$ and $\hat{F}$ restricts to a torsion-free sheaf on the generic fibre of $\hat{\pi}$;
 \end{itemize}
then $r(\hat{F})$ must be positive, and any torsion subsheaf of $\hat{F}$ must restrict to zero on the generic fibre, i.e.\ any torsion subsheaf of $\hat{F}$ lies in $\mathcal{B}_Y$.  Examples of sheaves satisfying property (G) above include torsion-free sheaves $F$ on $X$ with $\mu (F) < b/a$ such that $F$ restricts to a stable sheaf on the generic fibre of $\pi$: for such a sheaf $F$, it is $\Psi$-WIT$_1$ by Lemma \ref{lemma-FMTesLemma6-4}.  By \cite[Lemma 9.5]{BMef} and Lemma \ref{lemma-equiv1}, we deduce that $\hat{F}$ must restrict to a stable torsion-free sheaf on the generic fibre of $\hat{\pi}$.
\end{remark}

Combining Lemmas \ref{lemma-equiv3} and Remark \ref{remark5}, we obtain a criterion under which certain $\Psi$-WIT$_1$ torsion-free sheaves have torsion-free transforms:

\begin{theorem}\label{pro5}
Suppose $F$ is a coherent sheaf on $X$ satisfying property (G).  Then $\hat{F}$ is a torsion-free sheaf if and only if
\begin{equation}\label{eq20}
\Ext^1 (\mathcal B_X \cap W_{0,X}, F)=0.
\end{equation}
\end{theorem}

\begin{remark}\label{remark2}
On an elliptic surface $X$,  for a Chern character $(ch_0,ch_1,ch_2)=(r,\delta,n)$ where $r>0$ (i.e.\ rank) and $\delta f$ (i.e.\ fibre degree) are coprime, there is a polarisation with respect to which a torsion-free sheaf $F$ with Chern character $(r,\delta,n)$ on $X$ is $\mu$-stable if and only if its restriction to the generic fibre of $\pi$ is stable \cite[Proposition 7.1]{FMTes}.  With respect to such a polarisation, let $\MM$ denote the moduli space of stable torsion-free sheaves of Chern character $(r,\delta,n)$ on $X$.  Suppose $d := \delta f$ and $c := r$ in \eqref{eqn-rdunderPhi}.  Suppose, in addition, that $X$ is a relatively minimal elliptic surface and $a,b$ are the unique integers satisfying $br-ad=1$ and $0<a<r$.  We can consider the open subscheme
\[
  U := \{ F \in \MM : \hat{F} \text{ is torsion-free}\}
\]
of $\MM$.  By our choice of the polarisation on $X$ and our assumption on $a,b,r,d$, we have $d/r < b/a$; therefore, by \cite[Lemma 6.4]{FMTes}, every sheaf $F$ in $U$ is $\Psi$-WIT$_1$.  And in \cite[Section 7.2]{FMTes}, Bridgeland describes an open subscheme $V$ of $\text{Pic}^\circ (Y) \times \text{Hilb}^t (Y)$ such that $\Psi [1]$ takes $U$ isomorphically onto $V$.  This gives a birational equivalence between $\MM$ and $\text{Pic}^\circ (Y) \times \text{Hilb}^t (Y)$, which is the statement of \cite[Theorem 1.1]{FMTes}.  Theorem \ref{pro5} now allows us to describe  $U$  more directly, as the locus of all $F \in \MM$ satisfying the vanishing condition
\[
\Ext^1 (\mathcal B_X \cap W_{0,X}, F)=0.
\]
Moreover, by Lemma \ref{lemma11}, every $F \in U$ is a locally free sheaf.
\end{remark}

%
%
%
%

The following theorem gives a whole class of sheaves for which  the vanishing condition \eqref{eq20} in Theorem \ref{pro5} holds:

\begin{theorem}\label{theorem-main0}
Suppose  $\pi : X \to S$ is an elliptic threefold where all the fibres are Cohen-Macaulay curves with trivial dualising sheaves.  If $F$ is a $\Psi$-WIT$_1$ reflexive sheaf on $X$, then $F$ satisfies
\begin{equation*}
\Ext^1 (\mathcal B_X \cap W_{0,X}, F)=0.
\end{equation*}
\end{theorem}

\begin{remark}\label{remark4}
The reader would notice that, in the proof of Theorem \ref{theorem-main0} below, instead of assuming that $F$ is reflexive (besides being $\Psi$-WIT$_1$), it suffices to assume $F$ satisfies the following two properties:
\begin{itemize}
\item[(i)] The existence of a surjection \eqref{eq24}.
\item[(ii)] $F$ is locally free outside a codimension-3 locus on $X$.
\end{itemize}
Let us also denote
\begin{itemize}
\item[(i')] $F$ has homological dimension at most 1.
\end{itemize}
Then properties (i') and (ii) together imply property (i); this can be deduced from the spectral sequence
\begin{equation}\label{eq25}
  E_2^{p,q} = H^p (X, \EExt^q (G_1,G_2)) \Rightarrow \Ext^{p+q}(G_1,G_2)
\end{equation}
for coherent sheaves $G_1, G_2$ on $X$.  Also, for a torsion-free sheaf $F$ on a smooth projective threefold $X$, conditions (i') and (ii) together turn out to be equivalent to $F$ being reflexive.  To see this, suppose $F$ satisfies conditions (i') and (ii).  Since $F$ is torsion-free, its codimension is 0.  Since $F$ is assumed to have homological dimension at most 1, we have $\EExt^i (F,\omega_X)=0$ for $i \neq 0, 1$.  Moreover, the codimension of $\EExt^1 (F,\omega_X)$ is exactly 3.  Hence $F$ satisfies condition $S_{2,0}$ in the sense of Huybrechts \cite[Proposition 1.1.6]{HL}.  Finally, by \cite[Proposition 1.1.10]{HL}, condition $S_{2,0}$ is equivalent to reflexivity.  That the reflexivity of $F$ implies properties (i') and (ii) on a smooth projective threefold is well-known.
\end{remark}


Since Gorenstein varieties are exactly Cohen-Macaulay varieties whose dualising sheaves are  line bundles, all the fibres of $\pi$ in Theorem \ref{theorem-main0} are Gorenstein curves.

\begin{proof}[Proof of Theorem \ref{theorem-main0}]
We divide the proof into five steps.
\bigskip
\paragraph{Step 1.} We need to show that $\Ext^1 (A,F)$ vanishes for any $A \in \mathcal B_X \cap W_{0,X}$.  By Serre duality, we have $\Ext^1 (A,F) \cong \Ext^2 (F,A \otimes \omega_X)$.  Since $F$ is a reflexive sheaf on a threefold, we have a surjection (see \cite[Proposition 5]{Ver}, for instance):
 \begin{equation}\label{eq24}
  H^2 (X,\EExt^0 (F,A\otimes \omega_X)) \twoheadrightarrow \Ext^2 (F,A\otimes \omega_X).
 \end{equation}
 Therefore, it suffices to show that $H^2 (X,\EExt^0 (F,A \otimes \omega_X))$, i.e.\ $H^2 (X,\HHom (F,A \otimes \omega_X))$  vanishes.

  If the dimension of $A$ is at most 1, then $\HHom (F,A \otimes \omega_X)$ also has dimension at most 1, and  $H^2 (X, \HHom (F,A \otimes \omega_X))$  vanishes.  From now on, we assume that $A$ is supported in dimension 2.
\bigskip
\paragraph{Step 2.} We claim that it suffices to show the vanishing of $\Ext^1 (A,F)$ for any $A \in \mathcal B_X \cap W_{0,X}$ where the support of $\pi_\ast A$ is a reduced scheme: observe that $\Ext^1 (A,F)=0$ is equivalent to $\Hom (\hat{A},\hat{F})=0$, where $\hat{A} \in \mathcal B_Y \cap W_{1,Y}$.  Suppose $\supp{(\hat{A})}$ is not reduced.  Then there is some ideal sheaf $\II$ of $\OO_Y$ such that, if we write $D'_m$ to denote the closed subscheme of $Y$ defined by the ideal sheaf $\II^m$, then $D'_1$ is reduced, and $\supp{(\hat{A})}$ is contained in $D'_n$ for some positive integer $n$.  Now we perform induction on $n$.

Note that  $\hat{A}$ fits in a short exact sequence of coherent sheaves on $Y$
\begin{equation*}
0 \to K \to \hat{A} \to \hat{A}|_{D'_{1}} \to 0
\end{equation*}
where $\hat{A}|_{D'_{1}}$ also lies in $\mathcal B_Y$, while $K$ lies in $\mathcal B_Y \cap W_{1,Y}$ and is supported on $D'_{n-1}$.  By induction, $\Ext^1 (A,F)=0$ will follow from  the following two things:
\begin{itemize}
\item[(i)] $\Hom (\hat{A}|_{D_1'},\hat{F})=0$, and
\item[(ii)] $\Hom (\hat{A},\hat{F})=0$ when $n=1$ (the induction hypothesis).
\end{itemize}
  Note further that $\hat{A}|_{D'_1}$ itself fit in a short exact sequence of coherent sheaves
  \[
    0 \to A_0 \to \hat{A}|_{D'_1} \to A_1 \to 0
  \]
  where $A_i \in \mathcal B_Y \cap W_{i,Y}$ for $i=0,1$, and both $A_0, A_1$ are supported on $D_1'$.  Since $\hat{A_0}$ is a torsion sheaf, we have $\Hom (A_0,\hat{F}) \cong \Hom (\hat{A_0},F)=0$.  Hence (i) will follow from (ii), the induction hypothesis.  In other words, we can assume that $\hat{A}$ is supported on a reduced scheme.  Write $D' := \text{supp}(\hat{A})$.  Then the morphism $D' \to C := \text{supp}(\pi_\ast \hat{A})$ induced by $\hat{\pi} : Y \to S$ factors through the closed immersion $C_{red} \hookrightarrow C$.  Hence the support of $A$ itself is contained in the closed subscheme $X \times_S C_{red}$ of $X$.  And so, overall, to complete the proof of this theorem, we can assume that $A$ is supported on a 2-dimensional subscheme $D$ of $X$   (but the support of $A$ may not exactly be $D$) that fits in a fibre square
  \[
  \xymatrix{
   D \arinj[r]  \ar[d]^{\pi} & X \ar[d]^{\pi} \\
   C \arinj[r] & S
  }
  \]
  where we can assume that $C$ is a 1-dimensional reduced scheme, and we also write $\pi$ to denote the pullback morphism $D \to C$ by abuse of notation.

  \bigskip
\paragraph{Step 3.} To being with, note that $\pi : D \to C$ is both projective and flat (since $\pi : X \to S$ is so).
Now, we have
\begin{equation*}
 H^2 (X,\HHom (F,A \otimes \omega_X)) \cong H^2 (D,\bar{A})
\end{equation*}
where $\bar{A}$ is some coherent sheaf on $D$ such that $\iota_\ast \bar{A} = \HHom (F,A \otimes \omega_X)$.

The Leray spectral sequence applied to $\pi : D \to C$ gives us
\[
  E_2^{p,q} = H^p(C,R^q\pi_\ast (\bar{A})) \Rightarrow H^{p+q}(D,\bar{A}).
\]
Since all the fibres of $\pi$ are 1-dimensional, $R^q \pi_\ast (\bar{A})=0$ for all $q \neq 0, 1$ by \cite[Corollary III 11.2]{Harts}.  On the other hand, since $C$ is 1-dimensional, $E_2^{p,q}$ vanishes for $p \neq 0, 1$.  Hence $H^1(C,R^1\pi_\ast (\bar{A})) \cong H^2 (D,\bar{A})$, and it suffices for us to show that $H^1(C,R^1\pi_\ast (\bar{A}))$ vanishes.  Furthermore, it suffices to show that $R^1 \pi_\ast (\bar{A})$ is supported at a finite number of points.  That is, it suffices to show:
 \begin{equation}\label{eq23}
 \text{for a general closed point $s \in C$, we have } R^1 \pi_\ast (\bar{A}) \otimes k(s) =0.
 \end{equation}
Since $C$ is reduced, we can apply generic flatness \cite[Proposition 052B]{stacks}, and see that $\bar{A}$ is flat over an open dense subscheme of $C$.  Now, let $s \in C$ be a general closed point, $g$ be the fibre $\pi^{-1}(s)$, and $\bar{A}|_s$ be the (underived) restriction of $\bar{A}$ to the fibre $g$ over $s$.  By cohomology and base change \cite[Theorem III 12.11]{Harts}, we have
\[
R^1 \pi_\ast (\bar{A}) \otimes k(s) \cong H^1 (g,\bar{A} |_s).
\]
 The theorem would be proved if we can show that $H^1(g,\bar{A}|_s)=0$.
\bigskip
\paragraph{Step 4.} By our assumptions, the fibre $g := \pi^{-1}(s)$ is a projective Cohen-Macaulay curve with trivial dualising sheaf.  Therefore,
\begin{equation}\label{eqn-usingSDonCM}
  H^1(g,\bar{A}|_s) \cong \Ext^1_g (\OO_g, \bar{A}|_s) \cong \Hom_g (\bar{A}|_s, \OO_g)
\end{equation}
where the second isomorphism  follows from Serre duality.

Now, write $\hat{D} := C \times_S Y$.  Then we have a commutative diagram
\[
\xymatrix{
  D  \ar[drr] \ar[ddr] & \hat{D} \ar[drr] \ar[dd]& & \\
  & & X \ar[ddr]  & Y \ar[dd] \\
  & C \ar[drr] & & \\
  & & & S
}
\]
where the arrow $C \to S$ is a closed immersion, and the arrows $D \to X$ and $\hat{D} \to Y$ are its pullbacks.  Let us write $\iota$ to denote either the closed immersion $D \to X$ or $\hat{D} \to Y$.  Then $A = \iota_\ast \tilde{A}$ for some $\tilde{A}$ supported on $D$.  By the base change formula (see \cite[Proposition A.85]{FMNT} and also \cite[(6.3)]{FMNT}), we have $\Psi (\iota_\ast \tilde{A})=\iota_\ast\Psi_C (\tilde{A})$, which is a sheaf sitting at degree 0 since  $A$ is $\Psi$-WIT$_0$.  Here, $\Psi_C$ denotes the induced relative Fourier-Mukai transform from the derived category of $D$ to that of $\hat{D}$ over $C$.  And so  $\tilde{A}$ itself is a $\Psi_C$-WIT$_0$ sheaf on $D$.  Also, for a general closed point $s \in C$, we have $\tilde{A}|^L_s \cong \tilde{A}|_s$ by generic flatness, i.e.\ there is no need to derive the restriction; thus $\tilde{A}|^L_s \cong A|_s$ for a general closed point $s \in C$.

Since $\pi : D \to C$ is flat,  we have the isomorphism
\begin{equation}\label{eqn6}
 \Psi_s (\tilde{A}|_s^L) \cong (\Psi_C \tilde{A}) |^L_s
\end{equation}
by base change \cite[Proposition 6.1]{FMNT}; here, $\Psi_s$ denotes the induced Fourier-Mukai transform on the fibres $D(X_s) \to D(Y_s)$.

Putting all these together, we get, for a general closed point $s \in C$,
\[
 \Psi_s (A|_s) \cong \Psi_s (\tilde{A}|^L_s) \cong \hat{\tilde{A}}|^L_s \cong \hat{\tilde{A}} |_s
\]
where the last isomorphism follows from generic flatness.  Thus we see that, $A|_s$ is $\Psi_s$-WIT$_0$ for a general closed point $s \in C$.
\bigskip
\paragraph{Step 5.} Since $F$ is reflexive, it is locally free outside a 0-dimensional closed subset $Z$ of $X$. Let $\bar{V}$ denote the open subscheme $S \setminus \pi (Z)$ of $S$, and write $V := X \times_S \bar{V}$ and $\hat{V} := Y \times_S \bar{V}$.  Then $F$ is flat over $V$, and
\[
\Psi_{\bar{V}}(F|_V) \cong \Psi_{\bar{V}} (F|^L_V) \cong (\Psi F)|^L_{\hat{V}} \cong \hat{F}[-1] |_{\hat{V}},
\]
where we apply base change in the second isomorphism.  Thus $\hat{F}|_{\hat{V}}$ is $\Phi_{\bar{V}}$-WIT$_0$.

Now that we know $\hat{F}|_{\hat{V}}$ is $\Phi_{\bar{V}}$-WIT$_0$ and $\Phi_{\bar{V}}^0 (\hat{F}|_{\hat{V}}) \cong F|_V$ is flat over $\bar{V}$, we can apply \cite[Corollary 6.2]{FMNT} to obtain that $\hat{F}|_s^L$ is $\Phi_s$-WIT$_0$ for all $s \in \bar{V}$.  Since $\hat{F}|_{\hat{V}}$ is generically flat over $C \cap \bar{V}$ (which is an open dense subset of $C$), for a general closed point $s \in C$ we have $\hat{F}|_s^L \cong \hat{F}|_s$.  Therefore, for a general closed point $s \in C$, we have that $\hat{F}|_s$ is $\Phi_s$-WIT$_0$, and so $F|_s$ is $\Psi_s$-WIT$_1$.


Overall, for a general closed point $s \in C$, we have
\begin{align*}
  \Hom_g (\bar{A}|_s, \OO_g) & = \Hom_g ( \HHom (F,A \otimes \omega_X)|_s, \OO_g) \\
  &\cong \Hom_g (A|_s, F|_s),
\end{align*}
which must vanish  since $A|_s$ is $\Psi_s$-WIT$_0$ and $F|_s$ is $\Psi_s$-WIT$_1$.   This completes the proof of the theorem.
\end{proof}

Theorem \ref{theorem-main0} now gives rise to the following:

\begin{coro}\label{coro1}
Suppose $\pi : X \to S$ is an elliptic threefold where all the fibres are Cohen-Macaulay curves with trivial dualising sheaves.  Then, for any reflexive sheaf $F$ on $X$ with $\mu (F) < b/a$ such that its restriction to the generic fibre of $\pi$ is stable,  we have $F$ is $\Psi$-WIT$_1$, and $\hat{F}$ is torsion-free and stable with respect to some polarisation on $Y$.
\end{coro}

\begin{proof}
Take any reflexive sheaf $F$ as described.  That $F$ is $\Psi$-WIT$_1$ follows from Lemma \ref{lemma-FMTesLemma6-4}.  By Lemma \ref{lemma-equiv2}, $\hat{F}$ has nonzero rank, and so by Theorem \ref{pro5} and Theorem \ref{theorem-main0},  $\hat{F}$ is torsion-free.  That $\hat{F}$ is stable on $Y$ with respect to a suitable polarisation follows from \cite[Lemma 9.5]{BMef} (which also works for WIT$_1$ sheaves) and \cite[Lemma 2.1]{BMef}.
\end{proof}

C\u{a}ld\u{a}raru has a result that is somewhat similar: in \cite[Theorem 2]{Caldararu}, he shows that  for elliptic threefolds with relative Picard number 1, the Fourier-Mukai transform $\Psi$ takes fiberwise stable locally free sheaves with relatively prime degree and rank to fibrewise stable locally free sheaves.

The following lemma gives examples of WIT$_1$ torsion-free sheaves on $X$ whose transforms are not torsion-free:

\begin{lemma}\label{lemma23}
Suppose $\pi : X \to S$ is either an elliptic surface or an elliptic threefold.  If $Z \subset X$ is a 0-dimensional subscheme, and $I_Z$ its ideal sheaf, then for any line bundle $L$ on $X$ with $d(L)< b/a$, the sheaf $I_Z \otimes L$ is $\Psi$-WIT$_1$, and its transform $\widehat{I_Z \otimes L}$ has a nonzero torsion subsheaf.
\end{lemma}
\begin{proof}
In the short exact sequence
\[
 0 \to I_Z \otimes L \to L \to \OO_Z \to 0,
\]
the line bundle $L$ is $\Psi$-WIT$_1$ by Lemma \ref{lemma-FMTesLemma6-4}.  Hence $I_Z \otimes L$ is also $\Psi$-WIT$_1$.  Since $\OO_Z$ is 0-dimensional, it is $\Psi$-WIT$_0$, and $\Psi$ takes the above short exact sequence to the short exact sequence
\[
 0 \to \hat{\OO}_Z \to \widehat{I_Z \otimes L} \to \hat{L} \to 0,
\]
where $\hat{\OO}_Z$ is supported on a finite number of fibres, thereby proving the lemma.
\end{proof}

\begin{remark}
Note that,  consistent with Remark \ref{remark5}, the torsion subsheaf $\hat{\OO}_Z$ of $\widehat{I_Z \otimes L}$ lies in $\mathcal B_Y$.  Also, even though the ideal sheaf $I_Z$ is locally free outside a codimension-3 locus, its homological dimension is exactly two (see \cite[p.146]{OSS}), and so Theorem \ref{theorem-main0} does not apply.
\end{remark}

\section{Application 1: moduli of stable complexes}

Let $\pi : X\to S$ be an elliptic threefold, and $\hat{\pi} : Y \to S$ the Fourier-Mukai partner as in Section \ref{section-notation}.  In Theorem \ref{theorem-main2} in this section, we use the results in Section \ref{section-equiv} to show that there is  an open immersion from a moduli of stable complexes to a moduli space of Gieseker stable sheaves.  This gives us a moduli stack of stable complexes that admits a tame moduli space in the sense of Alper.

\subsection{An open immersion of moduli stacks}

Let us set up the notation: let $\sigma$ be any polynomial stability of type V2 in the sense of \cite{Lo3}, and $\sigma^\ast$ any polynomial stability of type V3 in the sense of \cite{Lo3}.  Let  $\MM^\sigma$ denote the moduli stack of $\sigma$-semistable objects in $D(X)$ of nonzero rank, while $\MM^{\sigma, \sigma^\ast}$ denote the substack of objects in $\MM^\sigma$ that are also $\sigma^\ast$-semistable.

For example, we can choose the stability function $p$ for $\sigma$ as $p(d)=-\lfloor \frac{d}{2}\rfloor$, and choose the stability vector $\rho$ so that $\rho_0, \rho_1, \rho_2, \rho_3$ are as in Figure 1 below (so that $\sigma$ is PT-stability, as in \cite{Lo1, Lo2}):

\begin{figure*}[h]
\centering
\setlength{\unitlength}{0.7mm}
\begin{picture}(170,40) 
\multiput(60,15)(1,0){50}{\line(1,0){0.5}}
\multiput(85,0)(0,1){40}{\line(0,1){0.5}}
\put(85,15){\vector(-4,1){15}}
\put(60,18.6){$-\rho_2$}
\put(85,15){\vector(-1,1){12}}
\put(67.5,28){$\rho_0$}
\put(85,15){\vector(1,2){10}}
\put(90,36){$-\rho_3$}
\put(85,15){\vector(3,1){14}}
\put(99.5,20){$\rho_1$}
\end{picture}

\caption{Configurations of the $\rho_i$ for PT-stability}
\label{figure-PTstab}
\end{figure*}
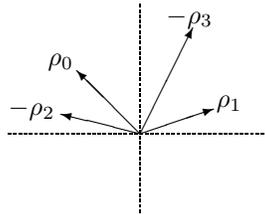
 In particular, every $\sigma$-semistable object $E$ in $D(X)$ is a 2-term complex such that $H^{-1}(E)$ is torsion-free, slope semistable and $H^0(E)$ is  0-dimensional.  
 
 To obtain an example of a polynomial stability of type V3, we can use the same stability function $p$ as above.  As for the stability vector $\rho$, we can simply switch the phases of $-\rho_3$ and $\rho_1$ in Figure \ref{figure-PTstab}; alternatively, we can switch the phases of $-\rho_3$ and $\rho_1$, as well as those of $\rho_0$ and $-\rho_2$ in Figure \ref{figure-PTstab} (see \cite[Section 2]{Lo3}).

For any Noetherian scheme $B$ over the ground field $k$ and any $B$-flat family of complexes $E_B$ on $X$, define the following property for fibres $E_b$ of $E_B$, $b \in B$:
    \begin{itemize}
      \item[(P)] The restriction $(H^{-1}(E_b))|_s$ of the cohomology sheaf $H^{-1}(E_b)$ to the fibre $\pi^{-1}(s)$ is a stable sheaf for a generic point $s \in S$.
    \end{itemize}
By Proposition \ref{pro3} below, property (P) is an open property for flat families of complexes on $X$.  Therefore, we have the following open immersions of moduli stacks:
\[
   \MM^{\sigma, \sigma^\ast, P} \subset \MM^{\sigma, \sigma^\ast} \subset \MM^\sigma.
\]
where $\MM^{\sigma, \sigma^\ast, P}$ denotes the stack of objects in $\MM^{\sigma, \sigma^\ast}$ that also have property (P).

Let $\MM^{\sigma, \sigma^\ast, P}_{\mu < b/a}$ denote the substack of $\MM^{\sigma, \sigma^\ast, P}$ consisting of complexes $E \in D(X)$ such that $\mu (H^{-1}(E)) < b/a$.

\begin{theorem}\label{theorem-main2}
Let $\pi : X \to S$ be as in Corollary \ref{coro1}.
We have an open immersion of moduli stacks
\[
\xymatrix{
\MM^{\sigma, \sigma^\ast, P}_{\mu < b/a}  \arinj[r]^(.6){} & \MM^s
}
\]
induced by the Fourier-Mukai transform $\Psi$, where $\MM^s$ denotes the moduli stack of Gieseker stable torsion-free sheaves on $Y$, with respect to some  polarisation.  Hence $\MM^{\sigma, \sigma^\ast, P}_{\mu < b/a}$ admits a tame moduli space in the sense of Alper.
\end{theorem}

\begin{proof}
Take any object $E \in D(X)$ corresponding to a point of $\MM^{\sigma, \sigma^\ast, P}$.  We know  $H^{-1}(E)$ is a reflexive sheaf from \cite[Lemma 3.2]{Lo4}, and that $H^0(E)$ is a 0-dimensional sheaf from \cite[Lemma 3.3]{Lo1}.  Having property (P) implies the restriction $H^{-1}(E)|_s$ of $H^{-1}(E)$ to a generic fibre $\pi^{-1}(s)$ is a stable sheaf.  By Corollary \ref{coro1}, we know $H^{-1}(E)$ is $\Psi$-WIT$_1$ and $\widehat{H^{-1}(E)}$ is a torsion-free sheaf.  The Fourier-Mukai transform $\Psi$ thus takes the canonical exact triangle in $D(X)$
\[
  H^{-1}(E)[1] \to E \to H^0(E) \to H^{-1}(E)[2]
\]
to the short exact sequence of coherent sheaves on $Y$
\[
0 \to \widehat{H^{-1}(E)} \to \hat{E} \to \widehat{H^0(E)} \to 0.
\]
Since $H^0(E)$ is supported at a finite number of points, it follows that $\widehat{H^0(E)}$ is supported on a finite number of fibres by base change \cite[Proposition 6.1]{FMNT}.  We also know that  the restriction of $\widehat{H^{-1}(E)}$ to a generic fibre of $\hat{\pi}$ is stable by \cite[Lemma 9.5]{BMef}.  Hence the restriction of $\hat{E}$ to a generic fibre of $\hat{\pi}$ is also stable.

Suppose $\hat{E}$ is not torsion-free; let $T$ be its maximal torsion subsheaf.  Since $\widehat{H^{-1}(E)}$ is torsion-free, we have an injection $T \hookrightarrow \widehat{H^0(E)}$.  On the other hand, since $H^0(E)$ is $\Psi$-WIT$_0$, $\widehat{H^0(E)}$ is $\Phi$-WIT$_1$; hence its subsheaf $T$ is also $\Phi$-WIT$_1$.  Now, the inclusion $T \subset \hat{E}$ gives us a nonzero element in
\begin{align*}
  \Hom_Y (T,\hat{E}) & \cong \Hom_X (\Phi T, \Phi \hat{E}) \\
  & \cong \Hom_X (\hat{T}[-1],E[-1]) \text{ since $\Phi \Psi \cong \text{id}_X[-1]$} \\
  &\cong \Hom_X(\hat{T},E).
\end{align*}
Since $\widehat{H^0(E)}$ is a sheaf supported on a finite number of fibres, so is $T$, and the same holds for $\hat{T}$ by base change.  Therefore, $\hat{T}$ is a sheaf supported in dimension at most 1.  By the definition of $\sigma^\ast$-stability \cite[Section 2]{Lo4}, however, there can be no nonzero morphisms from objects in $\Coh_{\leq 1}(X)$ to a $\sigma^\ast$-semistable object $E$.  Hence $T$ must be zero, i.e.\ $\hat{E}$ is torsion-free.

The last two paragraphs combined with \cite[Lemma 2.1]{BMef} give that $\hat{E}$ is torsion-free and stable with respect to a suitable polarisation $l$ on $Y$.  Using relative Fourier-Mukai transforms \cite[Section 6.1]{FMNT}, we can define a morphism of stacks $\MM^{\sigma, \sigma^\ast, P}_{\mu < b/a}  \hookrightarrow \MM^s$ induced by $\Psi$, where both $\MM^{\sigma, \sigma^\ast, P}_{\mu < b/a}$ and $\MM^s$ are contained as open substacks of the stack of relatively perfect universally gluable complexes constructed by Lieblich (see \cite{Lieblich} and \cite[Appendix]{ABL}).  That $\sigma$-semistability, $\sigma^\ast$-semistability and property (P) are all open properties for complexes \cite[Remark 4.4]{Lo3}, together with the fact that $\Psi$ is an equivalence, imply that this morphism of stacks is an open immersion.  Since  $\MM^s$ itself admits a tame moduli space \cite[Example 8.7]{Alper}, the open substack $\MM^{\sigma, \sigma^\ast, P}_{\mu < b/a}$ also admits a tame moduli space by \cite[Proposition 7.4]{Alper}.
\end{proof}

Note that $\MM^{\sigma, \sigma^\ast, P}_{\mu < b/a}$ contains as a substack the stack $\MM^{lf,P}_{\mu < b/a}$ of  locally free sheaves $F$ (sitting at degree $-1$) for which the restriction to the generic fibre of $\pi$ is torsion-free and slope semistable with $\mu  < b/a$.

\begin{remark}
The inclusion
\[
\MM^{lf,P}_{\mu < b/a} \subset \MM^{\sigma, \sigma^\ast, P}_{\mu < b/a}
\]
 is strict in general.  To see this, take any reflexive (or even locally free) sheaf $F$ on $X$ such that its restriction to the generic fibre of $\pi$ is stable with $\mu < b/a$.  Then for any short exact sequence of sheaves on $X$ of the form
\[
  0 \to F' \to F \to G\to 0
\]
where $G$ is supported on a hypersurface whose image under $\pi$ is 1-dimensional, $F'$ is still reflexive, but is not necessarily locally free \cite[Corollary 1.5]{SRS}.  If $F'$ is reflexive and non-locally free with relatively prime degree and rank, then we can produce an object $E$ in $\MM^{\sigma, \sigma^\ast}$ with $H^{-1}(E) \cong F'$ with nonzero $H^0(E)$ (so $E$ is not isomorphic to a sheaf) by \cite[Section 4.2]{Lo4}. Then, the restriction of $F'$ to the generic fibre of $\pi$ is again stable with $\mu < b/a$.  That is, $E$ is an object in $\MM^{\sigma, \sigma^\ast, P}_{\mu < b/a}$ but not $\MM^{lf,P}_{\mu < b/a} $.  Besides, from \cite[Section 6]{FMW} and \cite[Section 3.3]{CDFMR}, we know that torsion-free non-reflexive sheaves occur naturally in the construction of stable sheaves on elliptic threefolds.
\end{remark}

\begin{remark}
The conclusion of Theorem \ref{theorem-main2} can be strengthened if we choose a suitable polarisation $l$ on $X$ and an appropriate Chern character $ch$.  More precisely, let $\pi : X \to S$ be as in Theorem \ref{theorem-main2}.  Given a fixed Chern character $ch$ on $X$, suppose $l$ is a polarisation on $X$ satisfying the following property:
 \begin{itemize}
 \item[] For any coherent sheaf $F$ on $X$ with $ch_i(F)=ch_i$ for $i=0,1,2$ that is slope semistable with respect to $l$, the restriction of $F$ to a general fibre of $\pi$ is stable.
 \end{itemize}
Now, suppose $l$ is also the ample class used in the definition of either $\sigma$ or $\sigma^\ast$, where $\sigma, \sigma^\ast$ are as in Theorem \ref{theorem-main2}.  Then $H^{-1}(E)$ is slope semistable with respect to $l$ (by \cite[Lemma 3.3]{Lo1} and \cite[Lemma 3.2]{Lo3}) while $H^0(E)$ is 0-dimensional, and hence $E$ satisfies property (P) - in this case, the open immersion in Theorem \ref{theorem-main2} can be stated more simply as
\[
\xymatrix{
\MM^{\sigma, \sigma^\ast}_{\mu < b/a}  \arinj[r]^(.6){} & \MM^s.
}
\]
\end{remark}

\subsection{Openness of property (P)}

\begin{pro}\label{pro3}
Let $\pi : X \to S$ be an elliptic threefold.  The property (P) is an open property for a flat family of complexes in the category $\langle \Coh_{\leq 1}(X), \Coh_{\geq 3}(X)[1]\rangle$.
\end{pro}

Let $E_B$ be a $B$-flat family of complexes in $\langle \Coh_{\leq 1}(X), \Coh_{\geq 3}(X)[1]\rangle$, where $B$ is some Noetherian scheme.  To prove Proposition \ref{pro3}, it suffices to  show that the locus
\begin{equation}\label{eq10}
  W := \{ b \in B : E_b \text{ has property (P)} \}
\end{equation}
is a Zariski open set.  This is achieved by showing that $W$ is stable under generisation in Lemma \ref{lemma5}, and that $W$ is a constructible set in Lemma \ref{lemma6}.

\begin{lemma}\label{lemma5}
The set $W$ in \eqref{eq10} is stable under generisation.
\end{lemma}

\begin{proof}
To show that $W$ is stable under generisation, we can assume that $B = \Spec R$ is the spectrum of a discrete valuation ring $R$.  Let $\iota : \Spec k \hookrightarrow B$ and $j : \Spec K \hookrightarrow B$ be the closed immersion and open immersion of the closed point and the generic point of $B$, respectively.  Starting with the assumption that $L\iota^\ast E_B$ has property (P), we want to show that $j^\ast E_B$ also has property (P) (since $j^\ast$ is an exact functor, there is no need to derive it).

That $L\iota^\ast E_B$ has property (P) means that $H^{-1}(L\iota^\ast E_B) |_s$ is stable for a generic point $s \in S$.  Define the subset of $S$
\begin{multline*}
  U_1 := \{ s \in S : \text{ supp}(H^0 (L\iota^\ast E_B)) \cap \pi^{-1}(s) = \varnothing, \\
  \, H^{-1}(L\iota^\ast E_B)|_s \text{ is locally free} \}.
\end{multline*}
Since  $H^0 (L\iota^\ast E_B)$ is supported in dimension at most 1 by hypothesis, the locus of $s$ for which $\text{supp}(H^0(L\iota^\ast E_B))$ intersects nontrivially with $\pi^{-1}(s)$ is a closed subset of $S$ of dimension at most 1.  On the other hand, since $H^{-1}(L\iota^\ast E_B)$ is a torsion-free sheaf on $X$, it is locally free outside a 1-dimensional closed subset of $X$.  Hence the complement of $U_1$ is a closed subset of $S$ of dimension at most 1, i.e.\ $U_1$ is an open dense subset of $S$.

Since $E_B$ is a $B$-flat family of 2-term complexes, it is isomorphic to a 2-term complex on $X \times B$.  By the definition of $U_1$, for any $s \in U_1$, the exact triangle
\[
  H^{-1}(L\iota^\ast E_B)[1] \to L\iota^\ast E_B \to H^0(L\iota^\ast E_B) \to H^{-1}(L\iota^\ast E_B)[2] \text{\quad in $D(X_k)$}
\]
restricts to the exact triangle
\[
  H^{-1}(L\iota^\ast E_B)|^L_s [1] \to (L\iota^\ast E_B)|^L_s \to 0 \to  H^{-1}(L\iota^\ast E_B) |^L_s [2] \text{ in $D(X_s)$}.
\]
Hence for any $s \in U_1$, we have $(L\iota^\ast E_B)|^L_s \cong H^{-1}(L\iota^\ast E_B)|^L_s[1]$, which is isomorphic to the (shifted) underived restriction $H^{-1}(L\iota^\ast E_B)|_s [1]$ since $H^{-1}(L\iota^\ast E_B)$ is locally free on an open neighbourhood of $\pi^{-1}(s)$.  Hence
\[
  (L\iota^\ast E_B)|_{\pi^{-1}(U_1)} \cong H^{-1}(L\iota^\ast E_B)|_{\pi^{-1}(U_1)}[1],
\]
where $H^{-1}(L\iota^\ast E_B)|_{\pi^{-1}(U_1)}$ is an $U_1$-flat family of sheaves; in fact, it is a locally free sheaf on $\pi^{-1}(U_1)$.

We further define the subset of $U_1$
\[
  U_2 := \{ s \in U_1 : H^{-1}(L\iota^\ast E_B)|_s \text{ is a stable sheaf}\}.
\]
Since being stable is an open property for a flat family of sheaves, by the last paragraph, $U_2$ is an open subset in $U_1$.

Let us make some observations regarding the fibres of $E_B$ over $U_2$:
\begin{itemize}
\item[(a)] For any $s \in U_2$, we have $(H^0 (L\iota^\ast E_B))|_s=0$.  Since $H^0(L\iota^\ast E_B) \cong \iota^\ast H^0(E_B)$ (this uses the fact that $E_B$ has no cohomology higher than degree 0), we have $0 \cong (\iota^\ast H^0(E_B))|_s \cong \iota^\ast (H^0(E_B)|_s)$.  By semicontinuity, $j^\ast (H^0(E_B)|_s) =0$.  Hence $(j^\ast H^0(E_B))|_s$ vanishes, as does $(j^\ast H^0(E_B))|^L_s$.  From the exact triangle
    \[
    j^\ast H^{-1}(E_B)[1] \to j^\ast E_B \to j^\ast H^0(E_B) \to j^\ast H^{-1}(E_B)[2],
    \]
     we then obtain
     \begin{equation}\label{eq9}
       (j^\ast E_B)|^L_s \cong (j^\ast H^{-1}(E_B)[1])|^L_s \text{\quad for any $s \in U_2$}.
     \end{equation}
\item[(b)] For any $s \in S$, we have $(L\iota^\ast E_B)|^L_s \cong L\iota^\ast (E_B |^L_s)$.  If $s \in U_2$, then from above  we have $(L\iota^\ast E_B)|^L_s \cong H^{-1}(L\iota^\ast E_B)|^L_s[1]$, where $H^{-1}(L\iota^\ast E_B)|^L_s \cong H^{-1}(L\iota^\ast E_B)|_s$ is a stable locally free sheaf.
    As a result, for any $s \in U_2$, we have that $E_B|^L_s$ is a complex on $X_s \times B$ whose restriction to the central fibre over $B$ is a sheaf.  Hence $E_B|^L_s$ itself is a $B$-flat family of  sheaves (sitting at degree $-1$) on $X_s \times B$.  Then, since being stable  and being locally free are both open properties for a flat family of sheaves, $j^\ast (E_B |^L_s) \cong (j^\ast E_B)|^L_s$ is a stable locally free sheaf on $X_s \times \Spec K$.  Since this holds for any $s \in U_2$, we obtain that $(j^\ast E_B)|_{\pi^{-1}(U_2)}$ is an $U_2$-flat family of stable locally free sheaves sitting at degree $-1$.
\end{itemize}
Now, $(H^0 (j^\ast E_B))|_s \cong H^0 ( (j^\ast E_B)|^L_s)$, which is zero when $s \in U_2$ by \eqref{eq9}.  Hence $(H^0(j^\ast E_B))|^L_s \cong 0$ for $s \in U_2$.  Therefore, when we apply the restriction functor $-|^L_s$ (with $s \in U_2$) to the exact triangle
\[
 H^{-1}(j^\ast E_B)[1] \to j^\ast E_B\to H^0 (j^\ast E_B) \to H^{-1}(j^\ast E_B)[2] \text{ in $D(X \times \Spec K)$},
\]
we get $(j^\ast E_B) |^L_s \cong (H^{-1}(j^\ast E_B))|^L_s [1]$.  By observation (b) above, $(j^\ast E_B)|^L_s$ is a stable locally free sheaf at degree $-1$, for any $s \in U_2$.  Hence $H^{-1}(j^\ast E_B)$ is a $U_2$-flat family of stable locally free sheaves.  In other words, $j^\ast E_B$ also has property (P).  This shows that $W$ is stable under generisation.
\end{proof}

\begin{lemma}\label{lemma6}
The set $W$ in \eqref{eq10} is constructible.
\end{lemma}

\begin{proof}
We can assume that $B$ is of finite type over the ground field $k$.  In the proof of this lemma, let us use the following alternative description of $W$:
\begin{equation}
  W = \{ b \in B : \text{ the locus }\{ s \in S : H^{-1}(E_b)|_s \text{ is stable}\} \text{ has dimension 2}\}.
\end{equation}
By using a flattening stratification of $B$ for $H^{-1}(E_B)$ and $H^0(E_B)$, we can assume that the cohomology sheaves $H^{-1}(E_B), H^0(E_B)$ are both flat over $B$.  As a consequence, for any $b \in B$ we have $H^{-1}(E_b) \cong H^{-1}(E_B) |_b$.  And so
\[
  H^{-1}(E_b)|_s \cong H^{-1}(E_B)|_b |_s \cong H^{-1}(E_B) |_s |_b =: H^{-1}(E_B)|_{(s,b)}.
\]

Now, let $S \times B = \coprod_i T_i$ be a flattening stratification of $S \times B$ for $H^{-1}(E_B)$.  Then $H^{-1}(E_B)|_{T_i}$ is flat over $T_i$ for each $i$.


Let $\pi_S, \pi_B$ denote the projections from $S \times B$ to $S$ and $B$, respectively.  Define
\begin{multline*}
 W_i := \{ b \in \pi_B (T_i) : \\
 \{ s \in \pi_S (T_i) : H^{-1}(E_B) |_{T_i} |_{(s,b)} \text{ is stable}\} \text{ has dimension at least 2} \}.
\end{multline*}
It is straightforward to see that $W = \coprod_i W_i$.  Therefore, to show that $W$ is constructible, it is enough to show that each $W_i$ is constructible.  In other words, in order to show that $W$ is constructible, we can assume from now on that $H^{-1}(E_B)$ is flat over the entirety of $S \times B$.

Consider the set
\[
  U := \{ (s,b) \in S \times B : H^{-1}(E_B) |_{(s,b)} \text{ is stable}\}.
\]
Since being stable is an open property for a flat family of sheaves, $U$ is an open subset of $S \times B$.  Then the set
\[
  \tilde{W} := \{ (s,b) \in U : \text{ the fibre of $U$ over $\pi_B (b)$ has dimension at least 2}\}
\]
is a locally closed subset of $S \times B$ by semicontinuity, hence constructible.  Since $W = \pi_B (\tilde{W})$, we see that $W$ itself is also constructible.
\end{proof}

\section{An equivalence of categories}

Throughout this section, let $\pi : X \to S$ be an elliptic threefold satisfying the same assumptions as in Theorem \ref{theorem-main0}, and $\hat{\pi} : Y \to S$ its Fourier-Mukai partner as in Section \ref{section-notation}.

In Theorem \ref{theorem-main2}, we gave an open immersion
 \[
 \MM^{\sigma,\sigma^\ast,P}_{\mu < b/a} \hookrightarrow \MM^s
 \]
 from a moduli of stable complexes $\MM^{\sigma,\sigma^\ast,P}_{\mu < b/a}$ on $X$ to a moduli of stable sheaves $\MM^s$ on $Y$, induced by the functor $\Psi$.  In this section, we extend this open immersion to an isomorphism of stacks.

 Note that, if we want to map a moduli of two-term complexes $\MM$ (where some of the objects are not isomorphic to sheaves) into a moduli of sheaves via $\Psi [1]$, then not all the sheaves in the image $\Psi [1] (\MM)$ can be WIT$_0$.  We have:

 \begin{theorem}\label{theorem-main3}
The functor $\Psi$ induces a bijection between the following two sets:
 \begin{itemize}
 \item[(i)] the set $\mathcal C_X$ of objects $E$ in
 \[
 \langle \mathcal B_X\cap W_{0,X}, \mathcal B_X^\circ \cap W_{1,X} [1] \rangle
  \]
  satisfying
  \[
  \Hom (\mathcal B_X \cap W_{0,X},E)=0,
   \]
   such that $H^{-1}(E)$ has nonzero rank, $\mu (H^{-1}(E)) < b/a$, and $H^{-1}(E)$ restricts to a stable sheaf on the generic fibre of $\pi$;
 \item[(ii)] the set  $\mathcal C_Y$ of torsion-free sheaves $F$ on $Y$ such that $\mu (F) > -c/a$, and $F$ restricts to a stable sheaf on the generic fibre of $\hat{\pi}$, and such that in the unique short exact sequence
 \[
   0 \to A \to F \to B \to 0
 \]
 where $A$ is $\Phi$-WIT$_0$ and $B$ is $\Phi$-WIT$_1$, we have $B \in \mathcal B_Y$.  (Note that, this is equivalent to requiring $B$ to be  a torsion sheaf by Lemma \ref{lemma-E3-WIT1torsion}.)
 \end{itemize}
 Under the above bijection, we have $A = \widehat{H^{-1}(E)}$ and $B = \widehat{H^0(E)}$.
 \end{theorem}

 Note that, the category $\langle \mathcal B_X\cap W_{0,X}, \mathcal B_X^\circ \cap W_{1,X} [1] \rangle$ above is just the intersection of the  hearts $\langle \mathcal B_X, \mathcal B_X^\circ [1] \rangle$ and $\langle W_{0,X}, W_{1,X}[1]\rangle$ of two different t-structures on $D(X)$.  Also note that, the definitions of $\mathcal C_X$ and $\mathcal C_Y$ make no mention of any kind of stability.

\begin{proof}
Take an object $E$ in $\mathcal C_X$. Then $\Psi E = \hat{E}$ fits in the short exact sequence in $\Coh (Y)$
\[
 0 \to \widehat{H^{-1}(E)} \to \hat{E} \to \widehat{H^0(E)} \to 0
\]
where $\widehat{H^{-1}(E)}$ is $\Phi$-WIT$_0$ and $\widehat{H^0(E)}$ is $\Phi$-WIT$_1$.  From the definition of $\mathcal C_X$, we have $H^0(E) \in \mathcal B_X \cap W_{0,X}$, and so $\widehat{H^0(E)} \in \mathcal B_Y$.   That $\hat{E}$ has positive rank with $\mu > -c/a$ follows from Lemma \ref{lemma-equiv2}.

Since $H^{-1}(E)[1]$ is a subobject of $E$ in the heart $\langle \mathcal B_X, \mathcal B_X^\circ [1] \rangle$, the condition $\Hom (\mathcal B_X \cap W_{0,X},E)=0$ implies $\Hom (\mathcal B_X \cap W_{0,X}, H^{-1}(E)[1])=0$, i.e.\ $\Ext^1 (\mathcal B_X \cap W_{0,X},H^{-1}(E))=0$.  Thus, by Theorem \ref{pro5}, $\widehat{H^{-1}(E)}$ is torsion-free.

Now, suppose $\hat{E}$ itself is not torsion-free, and $T$ is its maximal torsion subsheaf.  Then $T \hookrightarrow \widehat{H^0(E)}$, and so $T \in \mathcal B_Y \cap W_{1,Y}$.  Thus $\hat{T} \in \mathcal B_X \cap W_{0,X}$.  Then $0=\Hom (\hat{T},E)\cong \Hom (T,\hat{E})$, a contradiction.  Hence $\hat{E}$ is torsion-free.

Conversely, suppose $F$ is a torsion-free sheaf in the  category $\mathcal C_Y$.  That the quasi-inverse $\Phi [1]$ of $\Psi$ takes $F$ into $\langle W_{0,X}, W_{1,X}[1]\rangle$ is clear.  The condition $B \in \mathcal B_Y$ (and knowing $B$ is $\Phi$-WIT$_1$) implies $H^0(\Phi [1] (F))=\hat{B} \in \mathcal B_X$.  On the other hand, that $A$ is $\Phi$-WIT$_0$ and torsion-free implies that $H^{-1} (\Phi [1] (F))=\hat{A}$ is torsion-free, by \cite[Lemma 9.4]{BMef}.  Hence $\Phi [1] (F) \in \langle \mathcal B_X, \mathcal B_X^\circ [1] \rangle$. Now, for any $T \in \mathcal B_X \cap W_{0,X}$, we have $\Hom (T,\Phi [1](F)) \cong \Hom (\hat{T},F)$, which vanishes because $\hat{T}$ is torsion and $F$ is torsion-free.  Lemma \ref{lemma-equiv2}  then completes the proof of the theorem.
\end{proof}

Let us  compare the equivalence in Theorem \ref{theorem-main3} to:
 \begin{itemize}
 \item the isomorphism between a connected component of the moduli of rank-one torsion-free sheaves on $X$ and a connected component of the moduli of stable torsion-free sheaves on $Y$ constructed by Bridgeland-Maciocia in \cite[Theorem 1.4]{BMef}, as well as
 \item the open immersion of moduli stacks in Theorem \ref{theorem-main2}.
 \end{itemize}

In the case of Bridgeland and Maciocia's result, they consider a moduli $\mathcal N$ of rank-one torsion-free sheaves on $Y$, all of which are $\Psi$-WIT$_0$ (after a suitable twist).  In terms of our notation in Theorem \ref{theorem-main3}, the sheaves parametrised by $\mathcal N$  are  exactly the rank-one torsion-free sheaves $F$ in $\mathcal C_Y$ with $B=0$, and they are taken by $\Phi [1]$ to objects in $\mathcal C_X$ with nonzero cohomology only at degree $-1$, which are torsion-free sheaves.

In the case of Theorem \ref{theorem-main2}, all the objects in $\MM^{\sigma, \sigma^\ast, P}_{\mu < b/a}$ lie in $\mathcal C_X$ by Lemma \ref{lemma22} below, and are taken to torsion-free sheaves $F$ in $\mathcal C_Y$ where $B$ is supported on a finite number of fibres of $\hat{\pi}$.

 \begin{lemma}\label{lemma22}
All the complexes corresponding to the closed points of $\MM^{\sigma, \sigma^\ast, P}_{\mu < b/a}$ in Theorem \ref{theorem-main2} lie in the category $\mathcal C_X$.
\end{lemma}

\begin{proof}
Given any complex $E \in D(X)$ corresponding to a closed point of $\MM^{\sigma, \sigma^\ast, P}_{\mu < b/a}$, we know that $\Psi (E)$ lies in the category $\mathcal C_Y$ in Theorem \ref{theorem-main3}, from the proof of Theorem \ref{theorem-main2}.  Hence by Theorem \ref{theorem-main3}, the complex $E$ lies in the category $\mathcal C_X$.
\end{proof}

\begin{remark}
Some of the sheaves in $\mathcal B_X \cap W_{0,X}$ are supported in dimension 2, while \textit{a priori}  we do not know that $\Hom (\Coh_{=2}(X),E)=0$ for $E \in \MM^{\sigma, \sigma^\ast, P}_{\mu < b/a}$.  Therefore, it is not immediately clear how  Lemma \ref{lemma22} can be shown with a direct proof (i.e.\  without considering the transforms on $Y$).
\end{remark}

\section{Application 2: pure codimension-1 sheaves}\label{section-app2}

In this section, we consider torsion-free sheaves on $X$ that are taken to torsion sheaves supported in codimension 1 by the Fourier-Mukai transform $\Psi  : D(X) \to D(Y)$.  When $X$ and $Y$ are elliptic surfaces, these torsion-free sheaves on $X$ are all locally free; if we further require them to be fiberwise semistable of fibre degree 0, then the corresponding torsion sheaves on $Y$ are all pure sheaves flat over the base $S$ (see Proposition \ref{pro4}).

Some of these results in this section resemble those obtained from the spectral construction of stable sheaves on elliptic fibrations (e.g.\ see \cite{CDFMR, FMW, HRP}), as well as results obtained by Yoshioka, where he assumes the existence of a section for the fibration (see \cite[Theorem 3.15]{Yoshioka}).

\begin{lemma}\label{lemma9}
Let $\pi : X \to S$ be an elliptic surface or threefold.  If $E$ is a torsion-free $\Psi$-WIT$_1$ sheaf on $X$, then $\hat{E}$ has no subsheaves of dimension 0.
\end{lemma}
\begin{proof}
Suppose $E$ is as above, and $\hat{E}$ has a nonzero  subsheaf $T$ of dimension 0.    Then we have a nonzero element in $\Hom_Y (T,\hat{E}) \cong \Hom_X (\hat{T},E)$, which is a contradiction because $E$ is torsion-free, and $\hat{T}$ is torsion.
\end{proof}

In the case of rank-one sheaves, we have a slight improvement of Lemma \ref{lemma-FMTesLemma6-4} with a different proof:

\begin{lemma}\label{lemma8}
Let $\pi : X \to S$ be either an elliptic  surface or threefold. Suppose $E$ is a rank-one torsion-free sheaf on $X$ with $\mu (E)=d(E)\leq b/a$.  Then $E$ is $\Psi$-WIT$_1$.
\end{lemma}

\begin{proof}
Consider the canonical exact sequence $0 \to E \to E^{\ast \ast} \to T \to 0$, where $T$ has codimension at least two, and so $d(T)=0$.  If $X$ is a threefold, then $E^{\ast \ast}$ is a rank-one reflexive sheaf, hence locally free, while  if $X$ is a surface, then any reflexive sheaf is locally free; in either case, $E^{\ast\ast}$ is locally free of rank one.  Write $L := E^{\ast \ast}$.  Then $E \cong L \otimes I_Z$, where $I_Z$ is the ideal sheaf of a closed subscheme $Z \subset X$ with codimension at least two.  Also, $d(E)=d(L)$.

Consider the short exact sequence $0 \to A \to E \to B \to 0$ where $A$ is $\Psi$-WIT$_0$ and $B$ is $\Psi$-WIT$_1$.  Suppose $A \neq 0$.  Then $A$ has rank one, so $B$ is a torsion sheaf.  By Lemma \ref{lemma-E3-WIT1torsion}, we have $B \in \mathcal B_X$.  Hence $d(B)=0$, and so $d(A)=d(E)=d(L)\leq b/a$.  Then $A$ is a $\Psi$-WIT$_0$ sheaf with positive rank and $\mu (A)\leq b/a$, contradicting Lemma \ref{lemma-PsiWIT0posrk}.  This implies $A$ must be zero, i.e.\ $L=E^{\ast\ast}$ is $\Psi$-WIT$_1$, and so its subsheaf $E$ is also $\Psi$-WIT$_1$.
\end{proof}
Note that, in the proof above, we do use the rank-one assumption on  $E$ in an essential way (in proving $B$ is torsion).

\begin{lemma}\label{lemma11}
Let $\pi : X \to S$ be an elliptic surface and $E$ a torsion-free sheaf on $X$ satisfying \[
  \Ext^1_{D(X)}(\mathcal B_X \cap W_{0,X}, E)=0.
\]
Then $E$ is a locally free sheaf.
\end{lemma}
\begin{proof}
Consider the canonical short exact sequence
\[
0 \to E \to E^{\ast\ast} \to T \to 0
\]
where $E^{\ast \ast}$ is reflexive, hence locally free, and $T$ is a 0-dimensional sheaf.  Applying the functor $\Hom (T,-)$ to this short exact sequence, we obtain the exact sequence
\[
0= \Hom (T,E^{\ast \ast}) \to \Hom (T,T) \to \Ext^1 (T,E).
\]
If $T$ is nonzero, then the identity map $1_T$ gives a nonzero element in $\Ext^1 (T,E)$.  However, $T \in \mathcal B_X \cap W_{0,X}$, so $\Ext^1 (T,E)=0$ by assumption.  Hence $T$ must have been zero to start with, i.e.\ $E$ is locally free.
\end{proof}

\begin{coro}\label{coro11}
Let $\pi : X \to S$ be an elliptic surface.  If $F$ is a $\Phi$-WIT$_0$ sheaf on $Y$ with no fibre subsheaves (i.e.\ $F \in \mathcal B_Y^\circ$), then $\hat{F}$ is a locally free sheaf on $X$.
\end{coro}
\begin{proof}
This follows from Lemma \ref{lemma11} and Lemma \ref{lemma-equiv3}.
\end{proof}

\begin{lemma}\label{lemma14}
Let $\pi : X \to S$ be an elliptic surface.  If $F$ is a pure 1-dimensional sheaf in $\mathcal B_Y^\circ$, then $F$ is $\Phi$-WIT$_0$.
\end{lemma}

\begin{proof}
Suppose $F$ is not $\Phi$-WIT$_0$.  Then, by \cite[Lemma 6.5]{FMTes}, there is a nonzero map $F \overset{\al}{\to} \mathcal Q_x$ for some $x \in X$.  Then $\image (\al)$ is a fibre sheaf, and it must be pure 1-dimensional, since $\mathcal Q_x$ is a stable sheaf supported on the fibre $\hat{\pi}^{-1}(\pi(x))$.  Having a surjection $F \twoheadrightarrow \image \al$ then implies $F$ contains the fibre $\hat{\pi}^{-1}(\pi(x))$ as a component of its support, which in turn implies $F$ has a nonzero fibre subsheaf, contradicting $F \in \mathcal B_Y^\circ$.  Hence $F$ must be $\Phi$-WIT$_0$.
\end{proof}

\begin{remark}\label{remark1}
If $F$ is a 1-dimensional sheaf on $Y$ that is flat over $S$, then the flatness implies the support of $F$ does not contain any fibre of $\hat{\pi}$.  If we also assume that $F$ is a pure sheaf, then $F$ has no 0-dimensional subsheaves.  Therefore, every pure 1-dimensional sheaf on $Y$ that is flat over $S$ lies in $\mathcal B_Y^\circ$, and is $\Phi$-WIT$_0$ by Lemma \ref{lemma14}.
\end{remark}

Given an elliptic fibration $\pi : X \to S$, it is a Weierstrass fibration in the sense of \cite[Definition 6.10]{FMNT} if it further satisfies:
\begin{itemize}
\item all the fibres of $\pi$ are geometrically integral Gorenstein curves of arithmetic genus 1;
\item there exists a section $\sigma : S \to X$ of $\pi$ such that its image $\sigma (S)$ does not contain any singular point of any fibre.
\end{itemize}
Let us call $\pi : X \to S$ a Weierstrass threefold (resp.\ surface) if $\pi$ is an elliptic threefold (resp.\ surface) that is also a Weierstrass fibration.

When $\pi : X \to S$ is a Weierstrass surface with $a=1$ and $b=0$, the Fourier-Mukai partner $\hat{\pi}: Y \to S$ is isomorphic to the Altman-Kleiman compactified relative Jacobian of $\pi$ \cite[Remark 6.33]{FMNT}.

\begin{pro}\label{pro4}
Let $\pi : X \to S$ be an elliptic surface.  The functor $\Psi [1] : D(X) \to D(Y)$ induces an equivalence of categories
\begin{multline*}
\{ \text{$\Psi$-WIT$_1$ torsion-free sheaves $E$ on $X$ satisfying $\mu (E)=b/a$}\} \to \\
\{ \text{$\Phi$-WIT$_0$ pure 1-dimensional, non-fibre sheaves $F$ satisfying } \\
\Hom (\mathcal B_Y\cap W_{0,Y},F)=0 \},
\end{multline*}
which restricts to the equivalence
\begin{multline*}
\{ \text{$\Psi$-WIT$_1$ locally free sheaves $E$ on $X$ satisfying  $\mu (E)=b/a$ and } \\
\Ext^1 (\mathcal B_X\cap W_{0,X},E)=0\} \to \\
\{ \text{pure 1-dimensional sheaves $F$ in $\mathcal B_Y^\circ$}\}.
\end{multline*}
If we further assume that $\pi : X \to S$ is a Weierstrass surface over $\Cfield$ with $a=1,b=0$, then the last equivalence further restricts to the equivalence
\begin{multline*}
  \{ \text{locally free, fiberwise semistable sheaves $E$ of fibre degree 0}\} \to \\
  \{ \text{pure 1-dimensional sheaves $F$, flat over  $S$}\}.
\end{multline*}
\end{pro}

\begin{proof}
The first equivalence follows immediately from Lemmas \ref{lemma15} and \ref{lemma-equiv1}.

For the second equivalence, given a sheaf $E$ on $X$ with the prescribed properties, that $\hat{E}$ on $Y$ has the desired properties follows from the first equivalence and Lemma \ref{lemma-equiv3}.  Conversely, if $F$ is a pure 1-dimensional sheaf on $Y$ that  is in $\mathcal B_Y^\circ$, then it is $\Phi$-WIT$_0$ by Lemma \ref{lemma14}.  By the first equivalence and Lemma \ref{lemma-equiv3}, we know that $\hat{F}$ is $\Psi$-WIT$_1$, with $\mu=b/a$ and satisfies $\Ext^1 (\mathcal B_X \cap W_{0,X},\hat{F})=0$.  Then Lemma \ref{lemma11} implies that $\hat{F}$ is locally free.  This shows the second equivalence.

For the last equivalence, we assume that $\pi : X \to S$ is a Weierstrass fibration and $a=1,b=0$. Suppose $E$ is a fiberwise semistable locally free sheaf of fibre degree 0.  Then for any fiber $\iota_s : X_s \hookrightarrow X$ of $\pi$, the restriction $E_s := \iota_s^\ast E$  is $\Psi_s$-WIT$_1$ by \cite[Proposition 6.51]{FMNT}.  Then, by \cite[Corollary 6.52]{FMNT}, $E$ itself is $\Psi$-WIT$_1$, and $\hat{E}$ is flat over $S$.

Now, we claim that $\Ext^1 (\mathcal B_X \cap W_{0,X},E)=0$.  To this end, we need to show $\Ext^1 (A,E)=0$ for any $A \in \mathcal B_X \cap W_{0,X}$, where it suffices to assume that $A$ is supported on a single fibre of $\pi$.  Take any such $A$, and suppose $ A = {\iota_s}_\ast \bar{A}$ for some sheaf $\bar{A}$ on the fibre $\iota_s : X_s \hookrightarrow X$.  Since all the fibres of $\pi$ are of dimension 1, while the kernel $\mathcal Q$ of the Fourier-Mukai transform $\Psi$ is a sheaf \cite[Lemma 5.1]{FMTes}, we have the base change
\begin{equation}\label{eq22}
  (\Psi^1 (A))_s \cong \Psi_s^1 (A|_s)
\end{equation}
by \cite[Corollary 6.3]{FMNT}.

Furthermore, since $\mathcal Q$ is flat over $X$ \cite[Lemma 5.1]{FMTes}, the kernel of the induced Fourier-Mukai transform $\Psi_s$ is a sheaf (i.e.\ $\mathcal Q_s$).  Then, because all the fibres of $\hat{\pi}$ are 1-dimensional, we have  $\Psi^i_s (A|_s)=0$ for $i>1$.   Since $A$ is assumed to be $\Psi$-WIT$_0$, we also have $\Psi_s^1 (A|_s)=0$ from \eqref{eq22}.  Hence $A|_s$ is $\Psi_s$-WIT$_0$. Now,
\begin{align}
  \Ext^1 (A,E) &\cong \Ext^1 (E,A \otimes \omega_X) \text{ by Serre duality} \notag\\
  &\cong H^1 (X,E^\ast \otimes A \otimes \omega_X)  \text{ since $E$ is locally free} \notag\\
  &\cong H^1 (X_s, (E_s)^\ast \otimes \bar{A} \otimes (\omega_X |_{X_s} ) )  \notag\\
  &\cong \Hom_{X_s} (\bar{A} \otimes (\omega_X |_{X_s}), E_s \otimes \omega_{X_s}) \text{ by Serre duality on $X_s$}. \label{eq21}
\end{align}
Since $\pi$ has a section, there are no multiple fibres of $\pi$, and so the formula for $\omega_X$ (see, for instance, \cite[(6.28)]{FMNT}) gives us $\omega_X |_{X_s}=\OO_{X_s}$ for any $s \in S$.  On the other hand, over $\Cfield$, all the fibres of a smooth elliptic surface have trivial dualising complexes \cite{genus1}, so $\omega_{X_s} = \OO_{X_s}$.  Then, since $\bar{A}$ is $\Psi_s$-WIT$_0$ and $E_s$ is $\Psi_s$-WIT$_1$, the Hom space \eqref{eq21} vanishes. Hence $\Ext^1 (\mathcal B_X \cap W_{0,X},E)=0$, and by the second equivalence, we get that $\hat{E}$ is pure 1-dimensional.

For the converse, suppose $F$ is a pure 1-dimensional sheaf on $Y$ that is flat over the base $S$.  By Remark \ref{remark1}, we know $F$ lies in $\mathcal B_Y^\circ$.  Using the second equivalence above, the only thing left to show is that $\hat{F}$ is fiberwise semistable.  Since we know $\hat{F}$ is $\Psi$-WIT$_1$ from the second equivalence, \cite[Proposition 6.51]{FMNT} implies that $\hat{F}$ is indeed fiberwise semistable.
\end{proof}

\begin{lemma}\label{lemma18}
Suppose $\pi : X \to S$ is an elliptic threefold or  surface.  Suppose $F$ is a pure codimension-1 sheaf on $Y$ that is flat over $S$.  Then $\hat{\pi}$ restricts to a finite morphism $\hat{\pi} : \text{supp}(F) \to S$, and $F \in \mathcal B_Y^\circ$. Furthermore, if $d(F)=1$, then $\text{supp}(F)$ is a section of $\hat{\pi} : Y \to S$, and $F$ is a line bundle on $\text{supp}(F)$.
\end{lemma}

\begin{proof}
Let $F$ be a pure codimension-1 sheaf that is flat over $S$.  Suppose $\text{supp}(F)$ contains a fibre $\hat{\pi}^{-1}(s)$ of $\hat{\pi}$.  Then the restriction $F|_s$ would be a sheaf of nonzero rank on $\hat{\pi}^{-1}(s)$, and by flatness, $F$ would be a sheaf of nonzero rank on $Y$, a contradiction.  Therefore, $\text{supp}(F) \cap \hat{\pi}^{-1}(s)$ is a finite set of points for any $s \in S$.  That is, the restriction $\hat{\pi} : \text{supp}(F) \to S$ is quasi-finite.  Since the closed immersion $\text{supp}(F) \hookrightarrow Y$ is projective and $\hat{\pi}$ itself is projective, the restriction $\hat{\pi} : \text{supp}(F) \to S$ is projective; since any projective, quasi-finite morphism is finite,  the restriction of $\hat{\pi}$ to $\text{supp}(F)$ is a finite morphism as claimed.

Now, suppose we can find a nonzero subsheaf $A \subset F$ such that $A \in \mathcal B_Y$.  Since the restriction  $\hat{\pi} : \text{supp}(F) \to S$ is quasi-finite, for any $s \in S$, the intersection $\text{supp}(A) \cap \hat{\pi}^{-1}(s)$ is a finite set of points.  On the other hand, that $A \in \mathcal B_Y$ implies $\hat{\pi} (\text{supp}(A))$ has codimension at least 1; this, along with the last sentence, implies $A$ has codimension at least 2 in $Y$.  Since $F$ is a pure sheaf, $A$ must be zero.  Hence $F \in \mathcal B_Y^\circ$, proving the first part of the lemma.

For the second part, assume $d(F)=1$.  Then for each $s \in S$, the fibre of $\hat{\pi}$ over $s$ intersects $\text{supp}(F)$ at one point with multiplicity 1.  Hence $\text{supp}(F)$ is flat over $S$ by \cite[Theorem III 9.9]{Harts}. Thus, locally, the morphism $\OO_S \to \hat{\pi}_\ast \OO_{\text{supp}(F)}$ makes $\hat{\pi}_\ast \OO_{\text{supp}(F)}$ a free module of rank 1 over $\OO_S$, implying $\OO_S \to \hat{\pi}_\ast \OO_{\text{supp}(F)}$ is surjective, hence an isomorphism.  Therefore, we obtain a section $\sigma : S \overset{\thicksim}{\to} \text{supp}(F)$.  Since $F$ is flat  over $S$, we see that  $F$ is a line bundle on $\text{supp}(F)$.
\end{proof}
When $d=1$, Lemma \ref{lemma18} also follows directly from \cite[Proposition 4.2]{HvdB}.

The last equivalence of Proposition \ref{pro4}, together with Lemma \ref{lemma18} and \eqref{eqn-rdunderPsi}, gives:

\begin{coro}\label{coro2}
Suppose $\pi : X \to S$ is a Weierstrass surface over $\Cfield$ and $a=1,b=0$.  Then the functor $\Psi [1]$ induces a bijection of sets
\begin{multline*}
  \{ \text{line bundles $E$ of fibre degree 0}\} \to \\
  \{ \text{$\sigma_\ast L$ :  \text{$\sigma$ is a section of $\hat{\pi}$, and $L \in \Pic (S)$}}\}.
\end{multline*}
\end{coro}

There are numerous results in existing literature that are similar to Corollary \ref{coro2}: see, for instance, \cite[Corollary 6.65]{FMNT}, \cite[Theorem 2.1]{HRP} and \cite[Theorem 3.15, Remark 3.6]{Yoshioka}.

\end{document}